\documentclass[a4paper]{article}

\usepackage[OT1]{fontenc}
\usepackage{amsthm,amsmath,amsfonts,amssymb,bbm,mathrsfs,mathtools}
\usepackage{geometry}
\usepackage{layout,authblk}

\usepackage{upref,algorithm2e}
\usepackage[shortlabels]{enumitem}
\usepackage{subcaption}
\usepackage{graphicx, caption, tikz}
\usepackage[colorlinks,citecolor=blue,urlcolor=blue]{hyperref}
\usepackage{multirow, tabularx, booktabs}
\usepackage{todonotes}
\usepackage{xcolor}

\newcommand{\prob}{\mathbb{P}}
\newcommand{\Prob}[1]{\prob\left(#1\right)}
\newcommand{\Probn}[1]{\prob_n\left(#1\right)}

\newcommand{\expec}{\mathbb{E}}
\newcommand{\Exp}[1]{\expec\left[#1\right]}

\newcommand{\Var}[1]{\textup{Var}\left(#1\right)}

\newcommand{\plim}{\ensuremath{\stackrel{\sss{\prob}}{\longrightarrow}}}

\newcommand{\ind}[1]{\mathbbm{1}{\left\{#1\right\}}}

\newcommand{\bigO}[1]{O\left(#1\right)}
\newcommand{\bigOp}[1]{O_{\sss\prob}\left(#1\right)}

\newcommand{\sss}[1]{\scriptscriptstyle{#1}}
\newcommand\abs[1]{\left|#1\right|}
\newcommand{\me}{\textup{e}}
\newcommand{\dd}{{\rm d}}

\newcommand{\op}{o_{\sss\prob}}

\newcommand{\URGnd}{{\rm URG}^{\sss{(n)}}(\boldsymbol{d})}

\newcommand{\Mn}{M^{\sss(n)}}
\newcommand{\Ecal}{{\mathcal{E}}}

\numberwithin{equation}{section}

\newcommand{\cs}[1]{{\color{black}#1}}

\allowdisplaybreaks

\newtheorem{theorem}{Theorem}[section]
\newtheorem{definition}{Definition}[section]
\newtheorem{lemma}[theorem]{Lemma}

\newtheorem{assumption}[definition]{Assumption}

\graphicspath{{Figures/}}
\numberwithin{equation}{section}
\setenumerate{label={\upshape{(\roman*)}}} 

\begin{document}
	\title{Distinguishing power-law uniform random graphs from inhomogeneous random graphs through small subgraphs}
\author{Clara Stegehuis}
\affil{Department of Electrical Engineering, Mathematics and Computer Science, University of Twente}
\maketitle
		
		\begin{abstract}
			We investigate the asymptotic number of induced subgraphs in power-law uniform random graphs. We show that these induced subgraphs appear typically on vertices with specific degrees, which are found by solving an optimization problem.
			Furthermore, we show that this optimization problem allows to design a linear-time, randomized algorithm that distinguishes between uniform random graphs and random graph models that create graphs with approximately a desired degree sequence: the power-law rank-1 inhomogeneous random graph. This algorithm uses the fact that some specific induced subgraphs appear significantly more often in uniform random graphs than in rank-1 inhomogeneous random graphs.
				\end{abstract}

\section{Introduction}
Many networks were found to have a degree distribution that is well approximated by a power-law distribution with exponent $\tau\in(2,3)$. These power-law real-world networks are often modeled by random graphs: randomized mathematical models that create networks. One of the most natural random graph models to consider is the uniform random graph~\cite{molloy1995,wormald1981}. Given a degree sequence, the uniform random graph samples a graph uniformly at random from all possible graphs with exactly that degree sequence. 


The most common way to analyze uniform random graphs, is to analyze the configuration model, another random graph model that is easier to analyze, instead~\cite{bollobas1980}. The configuration model creates random multigraphs with a specified degree sequence, i.e., graphs where multiple edges and self-loops can be present. When conditioning on the event that the configuration model results in a simple graph, it is distributed as a uniform random graph. If the probability of the event that the configuration model results in a simple graph is sufficiently large, it is possible to translate results from the configuration model to the uniform random graph. In the case of power-law degrees with exponent $\tau\in(2,3)$ however, the probability of the configuration model resulting in a simple random graph vanishes, so that the configuration model cannot be used as a method to analyze power-law uniform random graphs~\cite{janson2009c}. In this setting, uniform random graphs need to be analyzed directly instead. However, analyzing uniform random graphs is in general complex: the presence of edges are \emph{dependent}, and there is no simple algorithm for constructing a uniform random graph with power-law degrees.

Several other random graph models that are easy to generate, create graphs with \emph{approximately} the desired degree sequence. The most prominent such models are rank-1 inhomogeneous random graphs~\cite{britton2006,boguna2003,chung2002}. In these models, every vertex is equipped with a weight, and pairs of vertices are connected independently with a probability that is a function of the vertex weights. Another such model is the erased configuration model, which erases all multiple edges and self-loops in the configuration model~\cite{britton2006}. As these models are easy to generate and easy to analyze, they are often analyzed as a proxy for random graphs with a desired degree sequence.

In this paper, we investigate induced subgraphs in uniform random graphs. Several special cases of subgraph counts in uniform random graphs have been analyzed before, such as cycles~\cite{wormald1981,garmo1999}. However, existing results often need a bound on the maximal degree in the graph or the assumption that all degrees are equal, which does not allow for analyzing power-law random graphs with $\tau\in(2,3)$. Recently, triangles in uniform power-law random graphs have also been analyzed~\cite{gao2018}. In this paper, we investigate the subgraph count of \emph{all} possible induced subgraphs by using a recent method based on optimization models~\cite{hofstad2017d,garavaglia2018} which enabled to analyze subgraph counts in erased configuration models and preferential attachment models. We combine this method with novel estimates on the connection probabilities in uniform random graphs~\cite{gao2020} to obtain an optimization model that finds the most likely composition of an induced subgraph of a power-law uniform random graph. This method allows us to localize and enumerate all possible induced subgraphs.

We then use this optimization problem to design a randomized algorithm that distinguishes two types of rank-1 inhomogeneous random graphs from uniform random graphs in linear time. Interestingly, this shows that approximate-degree power-law random graphs are fundamentally different in structure from power-law uniform random graphs. Indeed, there are subgraphs that appear significantly more often in uniform random graphs than in these rank-1 inhomogeneous random graphs. Furthermore, the optimization problem that we use to prove results on the number of subgraphs allows to detect these differences in linear time, while subgraph counting in general cannot be done in linear time.

We first introduce the uniform random graph and the induced subgraph counts in Section~\ref{sec:urg}. Then, we present our main results on subgraph counts in the large network limit in Section~\ref{sec:main}. After that, we discuss the implications of these results for distinguishing uniform random graphs from inhomogeneous random graphs in Section~\ref{sec:algthm}. We then provide the proofs of our main results in Sections~\ref{sec:maxcont}-\ref{sec:algproof}.

\paragraph{Notation.}\label{sec:notation}
We denote $[k]=\{1,2,\dots,k\}$. We say that a sequence of events $(\mathcal{E}_n)_{n\geq 1}$ happens with high probability (w.h.p.) if $\lim_{n\to\infty}\Prob{\mathcal{E}_n}=1$ and we use $\plim $ for convergence in probability. We write $f(n)=o(g(n))$ if $\lim_{n\to\infty}f(n)/g(n)=0$, and $f(n)=O(g(n))$ if $|f(n)|/g(n)$ is uniformly bounded. We write $f(n)=\Theta(g(n))$ if $f(n)=O(g(n) )$ as well as $g(n)=O(f(n))$. We say that $X_n=O_{\sss{\prob}}(g(n))$ for a sequence of random variables $(X_n)_{n\geq 1}$ if $|X_n|/g(n)$ is a tight sequence of random variables, and $X_n=o_{\sss{\prob}}(g(n))$ if $X_n/g(n)\plim 0$.

	\paragraph{Uniform random graphs.}\label{sec:urg} Given a positive integer $n$ and a graphical \emph{degree sequence}: a sequence of $n$ positive integers $\boldsymbol{d}=(d_1,d_2,\ldots, d_n)$, the \emph{uniform random graph} ($\URGnd$) is a simple graph, uniformly sampled from the set of all simple graphs with degree sequence $(d_i)_{i\in[n]}$. Let 
	$d_{\max}=\max_{i\in[n]}d_i$ and $L_n=\sum_{i=1}^n d_i$. We denote the empirical degree distribution by
\begin{equation}
	F_n(j)=\frac{1}{n}\sum_{i\in[n]}\ind{d_i\leq j}.
\end{equation}

We study the setting where the variance of $\boldsymbol{d}$ diverges when $n$ grows large.	
In particular, we assume that the degree sequence satisfies the following assumption:
\begin{assumption}[Degree sequence]\label{ass:degrees}
	\leavevmode
	\begin{enumerate}
		\item \label{ass:degreeall}
		There exist $\tau\in(2,3)$ and constants $K_1,K_2>0$ such that for every $n\ge 1$ and every $0\le j\le d_{\max}$, 
		\begin{equation}\label{eq:bound}
			K_1 j^{1-\tau}\leq 1-F_n(j)\leq K_2 j^{1-\tau}.
		\end{equation}
		\item \label{ass:degreerange}
		There exist $\tau\in(2,3)$ and a constant $C>0$ such that, for all $j=O(\sqrt{n})$,
		\begin{equation}\label{D-tail}
			1-F_n(j) = Cj^{1-\tau}(1+o(1)).
		\end{equation}
	\end{enumerate}
\end{assumption}
It follows from~(\ref{eq:bound}) that
\begin{equation}\label{dmax}
	d_{\max} < M n^{1/(\tau-1)},\quad \mbox{for some sufficiently large constant $M>0$.}
\end{equation}
Furthermore, Assumptions~\ref{ass:degreeall} and~\ref{ass:degreerange} together show that
\begin{equation}\label{eq:mu}
	\lim_{n\to\infty}\frac{1}{n}\sum_{i=1}^nd_i=\lim_{n\to\infty}\frac{L_n}{n}=\mu<\infty,
\end{equation}
for some $\mu>0$.


\section{Main results}\label{sec:main}

We now present our main results. Let $H=(V_H,\Ecal_H)$ be a small, connected graph. We are interested in the induced subgraph count of $H$, the number of subgraphs of $\URGnd$ that are isomorphic to $H$.
Let $\URGnd|_{\boldsymbol{v}}$ denote the induced subgraph obtained by restricting $\URGnd$ to vertices $\boldsymbol{v}$. We can write the probability that an induced subgraph $H$ with $|V_H|=k$ is created on $k$ uniformly chosen vertices $\boldsymbol{v}=(v_1, \ldots, v_k)$ in $\URGnd$ as
\begin{equation}\label{eq:pHpresent}
\Prob{\URGnd|_{\boldsymbol{v}}= H}=\sum_{\boldsymbol{d}'}\Prob{\URGnd|_{\boldsymbol{v}}=H \mid d_{\boldsymbol{v}}=\boldsymbol{d}'}\Prob{d_{\boldsymbol{v}}=\boldsymbol{d}'},
\end{equation}
where the sum is over all possible degrees on $k$ vertices $\boldsymbol{d}'=(d_i')_{i\in [k]}$, and $d_{\boldsymbol{v}}=(d_{v_i})_{i\in[k]}$ denotes the degrees of the randomly chosen set of $k$ vertices. Recently, it has been shown that in erased configuration models, 
there is a specific range of $d_1',\ldots,d_k'$ that gives the maximal contribution to the amount of subgraphs of those degrees, sufficiently large to ignore all other degree ranges~\cite{hofstad2017d}. In this paper, we show that also~\eqref{eq:pHpresent} is optimized for specific ranges $d_1',\ldots,d_k'$ that depend on the subgraph $H$. 

Furthermore, we show that \cs{when~\eqref{eq:pHpresent} is maximized by a unique range of degrees}, there are only four possible ranges of degrees that maximize the term inside the sum in~\eqref{eq:pHpresent}. These ranges are constant degrees, or degrees proportional to $n^{(\tau-2)/(\tau-1)}$, to $\sqrt{n}$ or to $n^{1/(\tau-1)}$. Interestingly, these are the same ranges that contribute to the erased configuration model~\cite{hofstad2017d}. However, the optimal distribution of the subgraph vertices over these ranges may be different in the erased configuration model and the uniform random graph.

\subsection{Optimizing the subgraph degrees}
We now present the optimization problems that maximizes the summand in~\eqref{eq:pHpresent} for induced subgraphs. 
Let $H=(V_H,\Ecal_H)$ be a small, connected graph on $k\geq 3$ vertices. Denote the set of vertices of $H$ that have degree one inside $H$ by $V_1$. 
Let $\mathcal{P}$ be all partitions of $V_H\setminus V_1$ into three disjoint sets $S_1,S_2,S_3$.
 This partition into $S_1,S_2$ and $S_3$ corresponds to the optimal orders of magnitude of the degrees in~\eqref{eq:pHpresent}: $S_1$ is the set of vertices with degree proportional to $n^{(\tau-2)/(\tau-1)}$, $S_2$ the set with degrees proportional to $n^{1/(\tau-1)}$, and $S_3$ the set of vertices with degrees proportional to $\sqrt{n}$. We then derive an optimization problem that finds the partition of the vertices into these three orders of magnitude that maximizes the contribution to the number of induced subgraphs. When a vertex in $H$ has degree 1, its degree in $\URGnd$ is typically small, i.e., it does not grow {with} $n$.

Given a partition $\mathcal{P}=(S_1,S_2,S_3)$ of $V_H\setminus V_1$, let $\Ecal_{S_i}$ denote the set of edges in $H$ between vertices in $S_i$ and $E_{S_i}=|\Ecal_{S_i}|$ its size, $\Ecal_{S_i,S_j}$ the set of edges between vertices in $S_i$ and $S_j$ and $E_{S_i,S_j}=|\Ecal_{S_i,S_j}|$ its size, and finally $\Ecal_{S_i,V_1}$ the set of edges between vertices in $V_1$ and $S_i$ and $E_{S_i,V_1}=|\Ecal_{S_i,V_1}|$ its size. We now define the optimization problem that optimizes the summand in~\eqref{eq:pHpresent} as
	\begin{align}
	\label{eq:maxeqsub}
B(H) & =\max_{\mathcal{P}}\Big[\abs{S_1}+\frac{1}{\tau-1}\abs{S_2}(2-\tau-k+|S_1|+k_1)\nonumber\\
& \quad -\frac{2E_{S_1}-2E_{S_2}+E_{S_1,S_3}-E_{S_2,S_3}+E_{S_1,V_1}-E_{S_2,V_1}}{\tau-1}\Big].
	\end{align}

 Let $S_1^* ,S_2^* ,S_3^* $ be a maximizer of~\eqref{eq:maxeqsub}. Furthermore, for any $(\alpha_1,\ldots, \alpha_k)$ such that $\alpha_i\in[0,1/(\tau-1)]$, define
 	\begin{equation}\label{eq:Mnalph}
 	M_n^{(\boldsymbol{\alpha})}(\varepsilon)=\{ ({v_1,\ldots, v_k})\colon  d_{{v_i}}\in[\varepsilon,1/\varepsilon] (\mu n) ^{\alpha_i}\ \forall i\in[k] \}.
 	\end{equation}
These are the sets of vertices $(v_1,\ldots, v_k)$ such that ${d_{v_1}}$ is proportional to $n^{\alpha_1}$ and ${d_{v_2}}$ proportional to $n^{\alpha_2}$ and so on. 
Denote the number of subgraphs with vertices in $M_n^{(\boldsymbol{\alpha})}(\varepsilon)$ by $N (H,M_n^{(\boldsymbol{\alpha})}(\varepsilon))$. Define the vector $\boldsymbol{\alpha}$ as
 \begin{equation}\label{eq:alphasub}
	 {\alpha}_i =\begin{cases}
		 (\tau-2)/(\tau-1)& i\in S _1^*,\\
		 1/(\tau-1)& i\in S _2^*,\\
		 \tfrac{1}{2} & i\in S _3^*,\\
		 0 & i\in V_1.
	 \end{cases}
 \end{equation}

The next theorem shows that sets of vertices in $M_n^{\boldsymbol{\alpha} }(\varepsilon)$ contain a large number of subgraphs, and computes the scaling of the number of induced subgraphs:

\begin{theorem}[General induced subgraphs]\label{thm:motifs}
	Let $H$ be a subgraph on $k$ vertices such that the solution to~\eqref{eq:maxeqsub} is unique.
	\begin{enumerate}[(i)]
		\item 
For any $\varepsilon_n$ such that $\lim_{n\to\infty}\varepsilon_n=0$,
	\begin{equation}
	\frac{N \big(H,M_n^{(\boldsymbol{\alpha} )}\left(\varepsilon_n\right)\big) }{N (H)}\plim 1.
	\end{equation}
	\item Furthermore, for any fixed $0<\varepsilon<1$,\cs{
	\begin{equation}\label{eq:Nsubmag}
	\frac{N (H,M_n^{(\boldsymbol{\alpha} )}(\varepsilon))}{n^{\frac{3-\tau}{2}(k_{2+}+B (H))+k_1/2}} \leq  f(\varepsilon)+\op(1),
	\end{equation}
and
	\begin{equation}\label{eq:Nsubmaglow}
	\frac{N (H,M_n^{(\boldsymbol{\alpha} )}(\varepsilon))}{n^{\frac{3-\tau}{2}(k_{2+}+B (H))+k_1/2}} \geq  \tilde{f}(\varepsilon)+\op(1),
	\end{equation}
for some functions $f(\varepsilon),\tilde{f}(\varepsilon)<\infty $ not depending on $n$.} Here $k_{2+}$ denotes the number of vertices in $H$ of degree at least 2, and $k_1$ the number of degree-one vertices in $H$.
\end{enumerate}
\end{theorem}

Thus, Theorem~\ref{thm:motifs}(i) shows that asymptotically, all induced subgraphs $H$ have vertices in $M_n^{\boldsymbol{\alpha} }(\varepsilon)$, and Theorem~\ref{thm:motifs}(ii) then computes the scaling in $n$ of the number of such induced subgraphs.

Now we study the special class of induced subgraphs for which the unique maximum of \eqref{eq:maxeqsub} is $S_3^*=V_H$. 
By the above interpretation of $S_1^*$, $S_2^*$ and $S_3^*$, these are induced subgraphs where the maximum contribution to the number of such subgraphs is from vertices with degrees proportional to $\sqrt{n}$ in $\URGnd$. For such induced subgraphs, we can obtain the detailed asymptotic scaling including the leading constant:

\begin{theorem}[Induced subgraphs with $\sqrt{n}$ degrees]\label{thm:sqrtsub}
	Let $H$ be a connected graph on $k$ vertices with minimal degree 2 such that the solution to~\eqref{eq:maxeqsub} is unique, and $B (H)=0$. Then,
	\begin{equation}
	\frac{N (H)}{n^{\frac{k}{2}(3-\tau)}}\plim A (H)<\infty,
	\end{equation}
	with
	\begin{equation}\label{eq:Asub}
	A (H) = \left(\frac{C(\tau-1)}{\mu^{(\tau-1)/2}}\right)^k\!\!\int_{0}^{\infty}\!\cdots\! \int_{0}^{\infty}(x_1\cdots x_k)^{-\tau}\prod_{\mathclap{\{{i,j}\}\in \Ecal_{H}}}\frac{x_ix_j}{1+x_ix_j} \ \ \prod_{\mathclap{\{{u,v}\}\notin \Ecal_{H}}}\frac{1}{1+x_ux_v}\dd x_1\cdots \dd x_k.
	\end{equation}
\end{theorem}


\paragraph{Optimal induced subgraph structures.}
Interestingly, Theorem~\ref{thm:motifs} implies that the number of copies of a specific induced subgraph $H$ is dominated by the number of copies in which its vertices embedded in $\URGnd$ have specific degrees, determined by maximizing~\eqref{eq:maxeqsub}. First restricting to these degrees, and then analyzing the subgraph count, allows to obtain the scaling of the number of induced subgraphs in power-law uniform random graphs where the analysis method by using the configuration model breaks down. Furthermore, this does not only give us information on the total number of subgraphs, but also on where in the graph we are most likely to find them (i.e., on which degrees).

\paragraph{Automorphisms of $H$.}
An automorphism of a graph $H$ is a map $V_H\mapsto V_H$ such that the resulting graph is isomorphic to $H$. 
In Theorem~\ref{thm:sqrtsub} we count automorphisms of $H$ as separate copies, so that we may count multiple copies of $H$ on one set of vertices and edges. Therefore, to count the number of induced subgraphs without automorphisms, one should divide the results of Theorem~\ref{thm:sqrtsub} by the number of automorphisms of $H$.

\subsection{Distinguishing uniform random graphs from rank-1 inhomogeneous random graphs}\label{sec:algthm}
Uniform random graphs create random networks that are uniformly sampled from all graphs with precisely a desired degree sequence. However, in the power-law degree range with $\tau\in(2,3)$, it is difficult to generate such graphs, as the method that generates a configuration model until a simple graph is obtained does not work anymore~\cite{janson2009c}. Therefore, random graph models that generate networks with approximately a desired degree sequence are often used as a proxy for uniform random graphs instead, as many of these models are easy to generate. One such model is the rank-1 inhomogeneous random graph~\cite{chung2002,boguna2003}.
In the inhomogeneous random graph, every vertex $i$ is equipped with a weight $w_i$. We here assume that these weights are sampled from a power-law distribution with $\tau\in(2,3)$. Then, several choices of the connection probability $p(w_i,w_j)$ are possible. Common choices are~\cite{chung2002,boguna2003}
\begin{align}\label{eq:hvm}
	p(w_i,w_j)&=\min\Big(\frac{w_iw_j}{\mu n},1\Big),\\
\label{eq:ecm}
	p(w_i,w_j)&=\me^{-w_iw_j/(\mu n)},\\
\label{eq:irgurg}
	p(w_i,w_j)&=\frac{w_iw_j}{w_iw_j+\mu n},
\end{align}
where $\mu$ denotes the average weight. By choosing the connection probabilities in this manner, the degree of vertex $i$ is approximately $w_i$ with high probability~\cite{hofstad2017b}.

Theorems~\ref{thm:motifs} and~\ref{thm:motifs} indicate that in terms of induced subgraphs, the uniform random graph produces the same results as the rank-1 inhomogeneous random graph with connection probabilities as in~\eqref{eq:irgurg}. Intuitively, this can be seen from the constant $A(H)$ in Theorem~\ref{thm:sqrtsub}, where this connection probability appears in a scaled form. Indeed, our proofs are based on the fact that in the uniform random graph, the probability that two vertices form a connection can be approximated by~\ref{eq:irgurg} (see Lemma~\ref{lem:psuburg}). Therefore, Theorems~\ref{thm:motifs} and~\ref{thm:motifs} also hold for rank-1 inhomogeneous random graphs with connection probability~\eqref{eq:irgurg}.

In~\cite{hofstad2017d,stegehuis2019b}, similar theorems for random graphs with connection probability~\eqref{eq:ecm} and~\eqref{eq:hvm} were derived. The number of all induced subgraphs in the model with connection probability~\eqref{eq:hvm} has the same scaling in $n$ as the number of induced subgraphs in the models with connection probability~\eqref{eq:ecm}.  However, 
the scaling in $n$ of the number of copies of some induced subgraphs in models generated from~\eqref{eq:ecm} and~\eqref{eq:hvm} may be different from the scaling in the uniform random graph. The smallest such subgraphs are of size 6, and are plotted in Figure~\ref{fig:subdifs}. Figure~\ref{fig:subdifs} shows that these two subgraphs appear significantly more often in the uniform random graph than in the inhomogeneous random graphs.

Interestingly, this means that rank-1 inhomogeneous random graphs and uniform random graphs can be distinguished by studying small subgraph patterns of size 6. Previous results showed that random graphs with connection probability~\eqref{eq:irgurg} are distinguishable from those generated by connection probabilities~\eqref{eq:ecm} and~\eqref{eq:hvm} by their maximum clique size, which differs by a factor of $\log(n)$~\cite{janson2010}. However, finding the largest clique is an NP-hard problem~\cite{karp1972}, while this method only needs subgraphs of size 6 as an input, which can be detected in polynomial time. Furthermore, the difference between the amounts of the induced subgraphs of Figure~\ref{fig:subdifs} is not a logarithmic factor, but a polynomial factor, making it easier to detect such differences.

Specifically, we can show that in only $O(n)$ time, it is possible to distinguish between power-law uniform random graphs and the approximate-degree random graph models of~\eqref{eq:hvm} and~\eqref{eq:ecm} with high probability:

\begin{theorem}\label{thm:alg}
	There exists a randomized algorithm that distinguishes power-law uniform random graphs from power-law rank-1 inhomogeneous random graphs with connection probabilities~\eqref{eq:hvm} or~\eqref{eq:ecm} in time $O(n)$ with accuracy at least $1-n^\gamma\me^{-cn^{\beta}}$ for some $\gamma,\beta,c>0$.
\end{theorem}
We will prove this theorem in Section~\ref{sec:algproof}, where we also introduce the randomized algorithm that distinguishes between these two random graph models. This algorithm is based on the subgraph displayed in Figure~\ref{fig:subdifs} (c) and (d). It first selects vertices that have degrees close to $n^{1/(\tau-1)}$ and $n^{(\tau-2)/(\tau-1)}$, and then randomly searches among those vertices for the induced subgraph of Figure~\ref{fig:subdifs} (c). In a uniform random graph, this will be successful with high probability, whereas in the rank-1 inhomogeneous random graphs with connection probabilities~\eqref{eq:hvm} and~\eqref{eq:ecm} the algorithm fails with high probability.

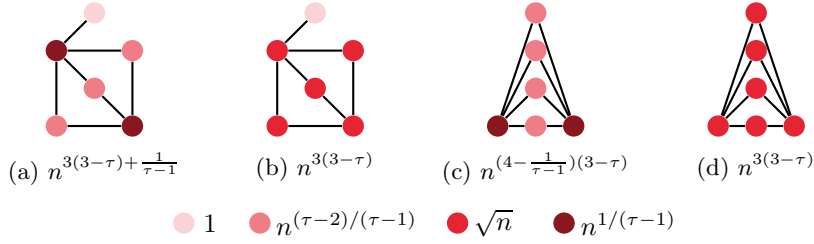
\begin{figure}[tb]
	\centering
	\definecolor{mycolor1}{RGB}{230,37,52}%
	\tikzstyle{every node}=[circle,fill=black!85,minimum size=8pt,inner sep=0pt]
	\tikzstyle{S1}=[fill=mycolor1!60]
	\tikzstyle{S2}=[fill=mycolor1!60!black]
	\tikzstyle{S3}=[fill=mycolor1]
	\tikzstyle{n1}=[fill=mycolor1!20]
	\begin{subfigure}[t]{0.2\linewidth}
	\centering
	\begin{tikzpicture}
		\tikzstyle{edge} = [draw,thick,-]
		\node[S1] (a) at (0,0) {};
		\node[S2] (b) at (1,0) {};
		\node[S2] (c) at (0,1) {};
		\node[S1] (d) at (1,1) {};
		\node[n1] (e) at (0.5,1.5) {};
		\node[S1] (f) at (0.5,0.5) {};
		\draw[edge] (a)--(b);
		\draw[edge] (c)--(a);
		\draw[edge] (d)--(b);
		\draw[edge] (c)--(e);
		\draw[edge] (c)--(f);
		\draw[edge] (b)--(f);
		\draw[edge] (c)--(d);
	\end{tikzpicture}	
	\caption{$n^{3(3-\tau)+\frac{1}{\tau-1}}$}
\end{subfigure}
	\begin{subfigure}[t]{0.2\linewidth}
	\centering
	\begin{tikzpicture}
		\tikzstyle{edge} = [draw,thick,-]
		\node[S3] (a) at (0,0) {};
		\node[S3] (b) at (1,0) {};
		\node[S3] (c) at (0,1) {};
		\node[S3] (d) at (1,1) {};
		\node[n1] (e) at (0.5,1.5) {};
		\node[S3] (f) at (0.5,0.5) {};
		\draw[edge] (a)--(b);
		\draw[edge] (c)--(a);
		\draw[edge] (d)--(b);
		\draw[edge] (c)--(e);
		\draw[edge] (c)--(f);
		\draw[edge] (b)--(f);
		\draw[edge] (c)--(d);
	\end{tikzpicture}	
	\caption{$n^{3(3-\tau)}$}
\end{subfigure}
	\begin{subfigure}[t]{0.2\linewidth}
	\centering
	\begin{tikzpicture}
		\tikzstyle{edge} = [draw,thick,-]
		\node[S2] (a) at (0,0) {};
		\node[S2] (c) at (1,0) {};
		\node[S1] (b) at (0.5,0) {};
		\node[S1] (d) at (0.5,1) {};
		\node[S1] (e) at (0.5,1.5) {};
		\node[S1] (f) at (0.5,0.5) {};
		\draw[edge] (a)--(b);
		\draw[edge] (c)--(b);
		\draw[edge] (a)--(d);
		\draw[edge] (c)--(d);
		\draw[edge] (a)--(f);
		\draw[edge] (c)--(f);
		\draw[edge] (c)--(e);
		\draw[edge] (a)--(e);
	\end{tikzpicture}	
	\caption{$n^{(4-\frac{1}{\tau-1})(3-\tau)}$}
\end{subfigure}
	\begin{subfigure}[t]{0.2\linewidth}
	\centering
	\begin{tikzpicture}
		\tikzstyle{edge} = [draw,thick,-]
		\node[S3] (a) at (0,0) {};
		\node[S3] (c) at (1,0) {};
		\node[S3] (b) at (0.5,0) {};
		\node[S3] (d) at (0.5,1) {};
		\node[S3] (e) at (0.5,1.5) {};
		\node[S3] (f) at (0.5,0.5) {};
		\draw[edge] (a)--(b);
		\draw[edge] (c)--(b);
		\draw[edge] (a)--(d);
		\draw[edge] (c)--(d);
		\draw[edge] (a)--(f);
		\draw[edge] (c)--(f);
		\draw[edge] (c)--(e);
		\draw[edge] (a)--(e);
	\end{tikzpicture}	
	\caption{$n^{3(3-\tau)}$}
\end{subfigure}

	\vspace{-0.5cm}
\begin{subfigure}{\linewidth}
	\centering
	\begin{tikzpicture}
		\node[S2,label={[label distance=0.05cm]0:$n^{1/(\tau-1)}$}] (a) at (5,0) {};
		\node[S3,label={[label distance=0cm]0:$\sqrt{n}$}] (b) at (3.6,0) {};
		\node[S1,label={[label distance=0.05cm]0:$n^{(\tau-2)/(\tau-1)}$}] (c) at (1,0) {};
		\node[n1,label={[label distance=0.05cm]0:$1$}] (c) at (0,0) {};
	\end{tikzpicture}
\end{subfigure}
\caption{Two induced subgraphs on 6 vertices with their scaling in $n$ and optimal degrees the uniform random graph from Theorem~\ref{thm:motifs} (Figures (a) and (c)) and inhomogeneous random graphs with connection probabilities~\eqref{eq:hvm} and~\eqref{eq:ecm} obtained from~\cite{hofstad2017d} (Figures (b) and (d)).}
\label{fig:subdifs}
\end{figure}

\paragraph{Organization of the proofs.}
We will prove Theorems~\ref{thm:motifs}-\ref{thm:alg} in the following sections. First, Section~\ref{sec:maxcont} proves Theorem~\ref{thm:motifs}(ii), by calculating the probability that $H$ appears on a specified subset of vertices, and optimizing that probability. Then, Section~\ref{sec:proof2} proves Theorem~\ref{thm:sqrtsub} with a second moment method. Section~\ref{sec:proofsec2} proves Theorem~\ref{thm:motifs}(i), and Section~\ref{sec:algproof} introduces and analyzes the randomized algorithm that proves Theorem~\ref{thm:alg}.

\section{Proof of Theorem~\ref{thm:motifs}(ii)}
\label{sec:maxcont}

\subsection{Subgraph probability in the uniform random graph}
We first investigate the probability that a given small graph $H$ appears as an induced subgraph of $\URGnd$ on a specific set of vertices $\boldsymbol{v}$. We denote the degree of a vertex $i$ inside induced subgraph $H$ by $d_i^{\sss{(H)}}$.

\begin{lemma}\label{lem:psuburg}
	Let $H$ be a connected graph on $k$ vertices, and let $\boldsymbol{d}$ be a degree sequence satisfying Assumption~\ref{ass:degrees}. Furthermore, assume that $d_{v_i}\gg 1$ or $d_i^{(H)}=1$ for all $i\in[k]$. Then,
	\begin{equation}
		\Probn{\URGnd|_{\boldsymbol{v}}=\Ecal_H} =\prod_{\{i,j\}\in \Ecal_H}\frac{d_{v_i}d_{v_j}}{L_n+d_{v_i}d_{v_j}}\prod_{\{s,t\}\notin \Ecal_H}\frac{1}{L_n+d_{v_s}d_{v_t}}(1+o(1))
	\end{equation}
\end{lemma}

\begin{proof}
	Suppose that $G^+$ is a subset of the edges of $H$, and $G^-$ a subset of the non-edges of $H$. Let $\mathcal{G}^-$ denote the event that the non-edges of $G^-$ are not present in $\URGnd|_{\boldsymbol{v}}$ and let $\mathcal{G}^+$ denote the event that the edges of $G^+$ are present in $\URGnd|_{\boldsymbol{v}}$.
	Let $d_{(1)}\geq d_{(2)}\geq \dots\geq d_{(n)}$ denote the ordered version of $\boldsymbol{d}$. Then, by Assumption~\ref{ass:degreeall}
	\begin{equation}
		d_{(k)}\leq W \left(\frac{n}{k}\right)^{1/(\tau-1)},
	\end{equation}
for some constant $W>0$. Therefore,
\begin{align}
	\sum_{i=1}^{Cn^{1/(\tau-1)}}d_{(i)} & \leq 	\sum_{i=1}^{Cn^{1/(\tau-1)}} W\left(\frac{n}{i}\right)^{1/(\tau-1)}\nonumber\\
	& \leq Wn^{1/(\tau-1)}\int_{0}^{Cn^{1/(\tau-1)}}x^{1/(1-\tau)}\dd x\nonumber\\
	& = \tilde{C}n^{1/(\tau-1)}n^{(\tau-2)/(\tau-1)^2}
\end{align}
for some $\tilde{C}>0$. Thus, $\sum_{i=1}^{Cn^{1/(\tau-1)}}d_{(i)}=o(n)$ for $\tau\in(2,3)$, while $L_n=\Theta(n)$ by~\eqref{eq:mu}. Therefore, we may apply~\cite[Corollary~2]{gao2020}
	to obtain
	\begin{equation}\label{eq:puvurg1}
		\Prob{\{i,j\}\in \URGnd \mid \mathcal{G}^-,\mathcal{G}^+}=\frac{(d_i-d_i^{(G)})(d_j-d_j^{(G)})}{L_n+(d_i-d_i^{(G)})(d_j-d_j^{(G)})}(1+o(1)),
	\end{equation}
where $d_i^{(G)}$ denotes the degree of vertex $i$ within $G^+$.
When $d_i\gg 1$, then $d_i-d_i^{(G)}=d_i(1+o(1))$ as $d_i^{(G)}\leq k-1$. Thus, when $d_i\gg1$ or $d_i^{(G)}=0$ and $d_j\gg1$ or $d_j^{(G)}=0$, then~\eqref{eq:puvurg1} becomes 
	\begin{equation}\label{eq:puvurg}
	\Prob{\{i,j\}\in \URGnd \mid \mathcal{G}^-,\mathcal{G}^+}=\frac{d_id_j}{L_n+d_id_j}(1+o(1)).
\end{equation}
Therefore also 
	\begin{equation}\label{eq:puvnonurg}
	\Prob{\{i,j\}\notin \URGnd \mid \mathcal{G}^-,\mathcal{G}^+}=\frac{1}{L_n+d_id_j}(1+o(1)).
\end{equation}
We now use~\eqref{eq:puvurg} and~\eqref{eq:puvnonurg} to compute the probability that $H$ appears as an induced subgraph on vertices $\boldsymbol{v}$. 
Let the $m$ edges of $H$ be denoted by $e_1={\{i_1,j_1\},\ldots,e_m=\{i_m,j_m\}}$, and the ${k\choose 2}-m$ non-edges of $H$ by $\bar{e}_1={\{w_1,z_1\},\ldots,\bar{e}_{k(k-1)/2-m}=\{w_m,z_m\}}$. Furthermore, define $G_0^+=\emptyset$ and $G_s^+=G_{s-1}^+\cup \{v_{i_s},v_{j_s}\}$. Similarly, define $G_0^-=\emptyset$ and $G_s^-=G_{s-1}^-\cup \{v_{w_s},v_{z_s}\}$. Then,
\begin{align}\label{eq:psuburg1}
	\Probn{\URGnd|_{\boldsymbol{v}}=\Ecal_H}& = \prod_{s=1}^m\Prob{\{v_{i_s},v_{j_s}\}\in \URGnd\mid G_{s-1}^+}\nonumber\\
	& \times \prod_{t=1}^{{k\choose 2}-m}\Prob{\{v_{w_s},v_{z_s}\}\notin \URGnd\mid G_m^+,G_{s-1}^-}.
\end{align}
We then use~\eqref{eq:puvurg} and~\eqref{eq:puvnonurg}. This is allowed because $d_{v_i}\gg 1$ or $d_i^{\sss{(H)}}=1$ for all $i\in[k]$ so that when the edge $e_l$ incident to vertex $i$ is added to in~\eqref{eq:psuburg1}, then $d_i^{\sss{(G_{l-1}^+)}}=0$. Indeed, when $d_i^{\sss{(H)}}=1$, then $i$ has no other incident edges in $H$, and therefore degree zero in $G_{l-1}^+$. Thus, we obtain 
\begin{equation}
	\Probn{\URGnd|_{\boldsymbol{v}}=\Ecal_H}=\prod_{\{i,j\}\in \Ecal_H}\frac{d_{v_i}d_{v_j}}{L_n+d_{v_i}d_{v_j}}\prod_{\{s,t\}\notin \Ecal_H}\frac{1}{L_n+d_{v_s}d_{v_t}}(1+o(1)).
\end{equation}
\end{proof}

\subsection{Optimizing the probability of a subgraph}\label{sec:opt}
We now study the probability that $H$ is present as an induced subgraph on vertices $(v_1, \ldots, v_k)$ of specific degrees. 
Assume that $d_{{v_i}}\in[\varepsilon,1/\varepsilon]n^{\alpha_i}$ with $\alpha_i\in[0,1/(\tau-1)]$ for $i\in[k]$, so that $d_{{v_i}}=\Theta(n^{\alpha_i})$.

Let $H$ be an induced subgraph on $k$ vertices labeled as $1,\ldots,k$. 
We now study the probability that $\URGnd|_{\boldsymbol{v}}=\mathcal{E}_H$.

 When $\alpha_i+\alpha_j< 1$, by Lemma~\ref{lem:psuburg}
\begin{equation}
	\Prob{X_{v_i,v_j}=0}=\Theta\Big(1-\frac{n^{\alpha_i+\alpha_j}}{n^{\alpha_i+\alpha_j}+\mu n} \Big)(1+o(1))=1+o(1),
\end{equation}
while 
\begin{equation}
	\Prob{X_{v_i,v_j}=1}=\Theta\Big(\frac{n^{\alpha_i+\alpha_j}}{n^{\alpha_i+\alpha_j}+\mu n} \Big)(1+o(1))=\Theta(n^{\alpha_i+\alpha_j-1}).
\end{equation}
On the other hand, for $\alpha_i+\alpha_j>1$,
\begin{equation}
	\Prob{X_{v_i,v_j}=0}=\Theta\Big(1-\frac{n^{\alpha_i+\alpha_j}}{n^{\alpha_i+\alpha_j}+\mu n} \Big)(1+o(1))=\Theta(n^{1-\alpha_i-\alpha_j}),
\end{equation}
while 
\begin{equation}
	\Prob{X_{v_i,v_j}=1}=\Theta\Big(\frac{n^{\alpha_i+\alpha_j}}{n^{\alpha_i+\alpha_j}+\mu n} \Big)(1+o(1))=1+o(1).
\end{equation}
Furthermore, when $\alpha_i+\alpha_j=1$, $\Prob{X_{v_i,v_j}=0}=\Theta(1)$ and $\Prob{X_{v_i,v_j}=1}=\Theta(1)$. 
Combining this with Lemma~\ref{lem:psuburg} shows that we can write the probability that $H$ occurs as an induced subgraph on $\boldsymbol{v}=(v_1,\cdots,v_k)$ as
\begin{equation}\label{eq:Ginduced}
	\begin{aligned}[b]
		&\Prob{\URGnd|_{\boldsymbol{v}}=\Ecal_H} = \Theta\bigg(\prod_{\{i,j\}\in \Ecal_H\colon \alpha_{i}+\alpha_{j}<1}\!\!\!\!\!\! n^{\alpha_{i}+\alpha_{j}-1} \ \ \ \prod_{\mathclap{\{u,v\}\notin \Ecal_H\colon 	\alpha_u+\alpha_v>1}} {n^{1-\alpha_u+\alpha_v}}\bigg).
	\end{aligned}
\end{equation}

Furthermore, by Assumption~\ref{ass:degrees} the number of vertices with degrees in $[\varepsilon,1/\varepsilon](\mu n)^\alpha$ is $\Theta(n^{(1-\tau)\alpha+1})$ for $\alpha\leq \frac{1}{\tau-1}$. Then, for $M_n^{(\boldsymbol{\alpha})}$ as in~\eqref{eq:Mnalph},
	\begin{equation}\label{eq:numhdeg}
	\# \text{ sets of vertices with degrees in }M_n^{(\boldsymbol{\alpha})}=\Theta( n^{k+(1-\tau)\sum_i\alpha_i}).
	\end{equation}
Thus,
	\begin{equation}\label{eq:Nalph}
	N (H,M_n^{(\boldsymbol{\alpha})}(\varepsilon))=\Theta_{\sss{\prob}}\Big( n^{k+(1-\tau)\sum_i\alpha_i} \ \ \prod_{\mathclap{\{i,j\}\in \Ecal_H:\alpha_i+\alpha_j<1}} \  \ \   n^{\alpha_{i}+\alpha_j-1}\ \  \ \prod_{\mathclap{\{u,v\}\notin \Ecal_H:\alpha_u+\alpha_v>1}} \ \  n^{-\alpha_{u}-\alpha_v+1}\Big).
	\end{equation}
Maximizing the exponent yields
\begin{equation}\label{eq:maxeqalphaind}
	\begin{aligned}[b]
		& \max_{\boldsymbol{\alpha}} (1-\tau)\sum_{i}\alpha_i +\ \  \ \ \ \sum_{\mathclap{\{i,j\}\in \Ecal_H\colon \alpha_i+\alpha_j<1}} \ \ \ (\alpha_i+\alpha_j-1)- \quad \ \sum_{\mathclap{\{u,v\}\notin \Ecal_H\colon \alpha_u+\alpha_v>1}} \ \ \ (\alpha_u+\alpha_v-1)
	\end{aligned}
\end{equation} 

The following lemma shows that this optimization problem attains its maximum for specific values of the exponents $\alpha_i$:

\begin{lemma}[Maximum contribution to subgraphs]\label{lem:maxmotif}
	Let $H$ be a connected graph on $k$ vertices. If the solution to~\eqref{eq:maxeqalphaind} is unique, then the optimal solution satisfies $\alpha_i\in\{0,\tfrac{\tau-2}{\tau-1},\tfrac{1}{2},\tfrac{1}{\tau-1}\}$ for all $i$. If it is not unique, then there exist at least 2 optimal solutions with $\alpha_i\in\{0,\tfrac{\tau-2}{\tau-1},\tfrac{1}{2},\tfrac{1}{\tau-1}\}$  for all $i$. In any optimal solution $\alpha_i=0$ if and only if vertex $i$ has degree one in $H$.
\end{lemma}

The proof of this lemma follows a similar structure as the proof of~\cite[Lemma 4.2]{hofstad2017d}, and we therefore defer it to Appendix~\ref{sec:prooflembeta}.
We now use the optimal structure of this optimization problem to prove Theorem~\ref{thm:motifs}(ii):

\begin{proof}[Proof of Theorem~\ref{thm:motifs}(ii)]
	
	Let $\boldsymbol{\alpha}$ be the unique optimizer of~\eqref{eq:maxeqalphaind}. 
	By Lemma~\ref{lem:maxmotif}, the maximal value of~\eqref{eq:maxeqalphaind} is attained by partitioning $V_H\setminus V_1$ into the sets $S_1,S_2,S_3$ such that vertices in $S_1$ have $\alpha_i=\tfrac{\tau-2}{\tau-1}$, vertices in $S_2$ have $\alpha_i =\tfrac{1}{\tau-1}$, vertices in $S_3$ have $\alpha_i=\tfrac{1}{2}$ and vertices in $V_1$ have $\alpha_i =0$. Then, the edges with $\alpha_i+\alpha_j <1$ are edges inside $S_1$, edges between $S_1$ and $S_3$ and edges from degree 1 vertices. Furthermore, non-edges with $\alpha_i+\alpha_j>1$ are edges inside $S_2$ (of which there are $\frac 12 |S_2|(|S_2|-1)-E_{S_2}$) or edges between $S_2$ and $S_3$ (of which there are $|S_2||S_3|-E_{S_2,S_3}$). Recall that the number of edges inside $S_1$ is denoted by $E_{S_1}$, the number of edges between $S_1$ and $S_3$ by $E_{S_1,S_3}$ and the number of edges between $V_1$ and $S_i$ by $E_{S_1,V_1}$. Then we can rewrite~\eqref{eq:maxeqalphaind} as
	\begin{equation}
		\label{eq:maxtemp}
		\begin{aligned}[b]
			\max_{\mathcal{P}} \ &\Big[(1-\tau) \left(\frac{\tau-2}{\tau-1}\abs{S_1}+\frac{1}{\tau-1}\abs{S_2}+\tfrac 12 \abs{S_3}\right)+\frac{\tau-3}{\tau-1}E_{S_1}\\
			&\qquad +\frac{\tau-3}{2(\tau-1)}E_{S_1,S_3}
			-\frac{E_{S_1,V_1}}{\tau-1}-\frac{\tau-2}{\tau-1}E_{S_2,V_1}-\frac 12 E_{S_3,V_1}\\
			& \qquad +\left(\frac 12 |S_2|(|S_2|-1)-E_{S_2}\right)\frac{\tau-3}{\tau-1}+(|S_2||S_3|-E_{S_2,S_3})\frac{\tau-3}{2(\tau-1)}\Big],
		\end{aligned}
	\end{equation}
	over all partitions $\mathcal{P}=(S_1,S_2,S_3)$ of $V_H\setminus V_1$. Using that $|S_3|=k-\abs{S_1}-\abs{S_2}-k_1$ and  ${E_{S_3,V_1}=k_1-E_{S_1,V_1}-E_{S_2,V_1}}$, where $k_1=\abs{V_1}$ and extracting a factor $(3-\tau)/2$ shows that this is equivalent to
	\begin{equation}
		\label{eq:maxtemp2}
		\begin{aligned}[b]
			&\frac{1-\tau}{2}k+	\max_{\mathcal{P}} \ \frac{(3-\tau)}{2}\Big( \abs{S_1}+\frac{1}{\tau-1}\abs{S_2}(2-\tau-k+|S_1|+k_1)+\frac{\tau-2}{3-\tau} k_1\\
			& \qquad \qquad \qquad-\frac{2E_{S_1}-2E_{S_2}+E_{S_1,S_3}-E_{S_2,S_3}}{\tau-1}  -\frac{E_{S_1,V_1}-E_{S_2,V_1}}{\tau-1}\Big).
		\end{aligned}
	\end{equation}
	Since $k$ and $k_1$ are fixed and $3-\tau>0$, we need to maximize
	\begin{align}
		\label{eq:maxeq}
		B(H)&=\max_{\mathcal{P}}\Big[\abs{S_1}+\frac{1}{\tau-1}\abs{S_2}(2-\tau-k+|S_1|+k_1) \nonumber\\
		& \quad -\frac{2E_{S_1}-2E_{S_2}+E_{S_1,S_3}-E_{S_2,S_3}+E_{S_1,V_1}-E_{S_2,V_1}}{\tau-1}\Big],
	\end{align}
	which equals \eqref{eq:maxeqsub}.

	By~\eqref{eq:Nalph}, the maximal value of $	N (H,M_n^{(\boldsymbol{\alpha})}(\varepsilon))$ then scales as 
	\begin{equation}\label{eq:maxcontrscaling}
		n^{\frac{3-\tau}{2}(k+B (H))+\frac{\tau-2}{2}k_1}=n^{\frac{3-\tau}{2}(k_{2+}+B (H))+{k_1/2}} ,
	\end{equation}
	which proves Theorem~\ref{thm:motifs}(ii). 
\end{proof}

\section{Proof of Theorem~\ref{thm:sqrtsub}}\label{sec:proof2}
In this section, we will prove Lemma~\ref{lem:convNH} that is given below, from which we prove Theorem~\ref{thm:sqrtsub}. For that, we define the special case of $M_n^{\sss{(\boldsymbol{\alpha})}}(\varepsilon)$ of~\eqref{eq:Mnalph} where $\alpha_i=\tfrac{1}{2}$ for all $i\in V_H=[k]$ as
	\begin{equation}
		W_n^k(\varepsilon)=\{(v_1,\ldots,v_k)\colon d_{{v_s}}\in[\varepsilon,1/\varepsilon]\sqrt{\mu n} \quad \forall s \in[k]\},
	\end{equation}
\cs{and let $\bar{W}_n^k(\varepsilon)$ denote the complement of $W_n^k(\varepsilon)$.}
	We denote the number of subgraphs $H$ with all vertices in $W_n^k(\varepsilon)$ by $N (H,W_n^k(\varepsilon))$.

\begin{lemma}[Major contribution to subgraphs]\label{lem:convNH}
	Let $H$ be a connected graph on $k{\geq 3}$ vertices such that~\eqref{eq:maxeqsub} is uniquely optimized at $S_3=[k]$, so that $B (H)=0$. Then,
	\begin{enumerate}[(i)]
	\item \label{lem:convNH1} the number of subgraphs with vertices in $W_n^k(\varepsilon)$ satisfies
	\begin{align}
		\frac{N (H,W_n^k(\varepsilon))}{n^{\frac{k}{2}(3-\tau)}} 
		\to & (C(\tau-1))^k\mu^{-\frac{k}{2}(\tau-1)} \int_{\varepsilon}^{1/\varepsilon}\!\!\cdots \int_{\varepsilon}^{1/\varepsilon}(x_1\cdots x_k)^{-\tau}\nonumber\\
		& \times \prod_{\mathclap{\{i,j\}\in \Ecal_H}}\frac{x_ix_j}{1+x_ix_j} \ \ \prod_{\mathclap{\{u,v\}\notin \Ecal_H}}\frac{1}{1+x_ux_v}\dd x_1\cdots \dd x_k .
	\end{align}
	\item \label{lem:Afinite}
$A (H)$ defined in~\eqref{eq:Asub} satisfies $A (H)<\infty$.
	\end{enumerate}
	\end{lemma}

We now prove Theorem~\ref{thm:sqrtsub} using this lemma. 

\begin{proof}[Proof of Theorem~\ref{thm:sqrtsub}]
	We first study the expected number of induced subgraphs with vertices outside $W_n^k(\varepsilon)$ and show that their contribution to the total number of copies of $H$ is small. First, we investigate the expected number of copies of $H$ in the case where vertex 1 of the subgraph has degree smaller than $\varepsilon\sqrt{\mu n}$. 
	By Lemma~\ref{lem:psuburg}, the probability that $H$ is present on a specified subset of vertices $\boldsymbol{v}=(v_1,\ldots,v_k)$ satisfies
		\begin{align}
			\Prob{\URGnd|_{\boldsymbol{v}}= \Ecal_H} & =\Theta \Big( \prod_{\{i,j\}\in \Ecal_H}\frac{d_{v_i}d_{v_j}}{L_n+d_{v_i}d_{v_j}} \prod_{\{u,w\}\notin \Ecal_H}\frac{L_n}{L_n+d_{v_u}d_{v_w}}\Big)
		\end{align}

	Furthermore, by~\eqref{D-tail}, there exists $C_0$ such that $\Prob{D=k}\leq C_0k^{-\tau}$ for all $k$, where $D$ denotes the degree of a uniformly chosen vertex. Let $I(H,\boldsymbol{v})=\ind{\URGnd|_{\boldsymbol{v}}= \Ecal_H},$ so that $N(H)=\sum_{\boldsymbol{v}} I(H,\boldsymbol{v})$. 
	
	Define
	\begin{equation}\label{eq:gn}
		h_n(x_1,\dots,x_k)=\prod_{\mathclap{\{i,j\}\in \Ecal_H}} \ \frac{x_ix_j}{\mu n+x_ix_j} \ \ \prod_{\mathclap{\{s,t\}\notin \Ecal_H}}  \ \frac{\mu n}{\mu n+x_sx_t}.
	\end{equation}

We can use similar methods as in~\cite[Eq.~(4.4)]{gao2018} to show that for some $K^*>0$,
\begin{align}\label{eq:expdeg1small1}
	&\sum_{\boldsymbol{v}}\Exp{I(H,\boldsymbol{v})\ind{d_{v_1}<\varepsilon\sqrt{\mu n}}}\nonumber\\
	& = n^k\int_{1}^{\varepsilon\sqrt{\mu n}}\int_{1}^{\infty}\cdots\int_{1}^{\infty}h_n(x_1,x_2,\dots,x_n) \dd F_n(x_k)\dots \dd F_n(x_1) \nonumber\\
	& \leq n^k K^* \int_{1}^{\varepsilon\sqrt{\mu n}}\int_{1}^{\infty}\cdots\int_{1}^{\infty}(x_2\cdots x_k)^{-\tau} h_n(x_1,x_2,\dots,x_n) \dd x_k\dots \dd F(x_1) .
\end{align}

For all non-decreasing $g$ that are bounded on $[0,\varepsilon \sqrt{\mu n}]$ and once differentiable, where $\bar{G}(x)$ denotes a function such that $\int_0^x\bar{G}(y)\dd y=g(x)$
\begin{align}
	&\int_{0}^{\varepsilon\sqrt{\mu n}}g(x)\dd F_n(x)= \int_{0}^{\varepsilon\sqrt{\mu n}}\int_0^{x}\bar{G}(y)\dd y\dd F_n(x)\nonumber\\
	 &=  \int_{0}^{\varepsilon\sqrt{\mu n}}(F_n(\varepsilon\sqrt{\mu n})-F_n(y))\bar{G}(y)\dd y\nonumber\\
	&=   C\left(\int_{0}^{\varepsilon\sqrt{\mu n}}y^{1-\tau}\bar{G}(y)\dd y-\int_{0}^{\varepsilon\sqrt{\mu n}}(\varepsilon\sqrt{\mu n})^{1-\tau}\bar{G}(y)\dd y\right)(1+o(1))\nonumber\\
	& =C\left((\tau-1) \int_{0}^{\varepsilon\sqrt{\mu n}}y^{-\tau}g(y)\dd y +\big[cg(y)y^{1-\tau}\big]_0^{\varepsilon\sqrt{\mu n}}-(\varepsilon\sqrt{\mu n})^{1-\tau}g(\varepsilon\sqrt{\mu n})\right)(1+o(1))\nonumber\\
	& =C(\tau-1) \int_{0}^{\varepsilon\sqrt{\mu n}}y^{-\tau}g(y)\dd y +o((\varepsilon\sqrt{\mu n})^{1-\tau}g(\varepsilon\sqrt{\mu n})),
\end{align}
where we have used Assumption~\ref{ass:degrees}\ref{ass:degreerange}.
Taking 
\begin{equation}
	g(x)=g_n(x)= \int_{1}^{\infty}\cdots\int_{1}^{\infty}(x_2\cdots x_k)^{-\tau} h_n(x,x_2,\dots,x_n) \dd x_2\dots \dd x_k
\end{equation}
 yields for~\eqref{eq:expdeg1small1}
\begin{align}\label{eq:expdeg1small2}
	&\sum_{\boldsymbol{v}}\Exp{I(H,\boldsymbol{v})\ind{d_{v_1}<\varepsilon\sqrt{\mu n}}}\nonumber\\
	& \leq n^k K^* \int_{1}^{\varepsilon\sqrt{\mu n}}\int_{1}^{\infty}\cdots\int_{1}^{\infty}(x_1\cdots x_k)^{-\tau} h_n(x_1,x_2,\dots,x_n) \dd x_1\dots \dd x_k \nonumber \\
	& \quad + o\left(n^k (\varepsilon \sqrt{\mu n})^{1-\tau} \int_{1}^{\infty}\cdots\int_{1}^{\infty}(x_2\cdots x_k)^{-\tau} h_n(\varepsilon \sqrt{\mu n},x_2,\dots,x_n) \dd x_2\dots \dd x_k \right).
\end{align}

Now we can bound the first term of~\eqref{eq:expdeg1small2} as
	\begin{equation}
		\begin{aligned}[b]
			&n^k\int_{1}^{\varepsilon\sqrt{\mu n}}\int_{1}^{\infty}\cdots\int_{1}^{\infty}(x_1\cdots x_k)^{-\tau} \ \prod_{\mathclap{\{i,j\}\in \Ecal_H}} \ \frac{x_ix_j}{\mu n+x_ix_j} \ \ \prod_{\mathclap{\{u,w\}\notin \Ecal_H}}  \ \frac{\mu n}{\mu n+x_ux_w}\dd x_1\cdots \dd x_k\\
			&=n^k(\mu n)^{\frac{k}{2}(1-\tau)} \int_{0}^{\varepsilon}\int_{0}^{\infty}\cdots\int_{0}^{\infty}(t_1\cdots t_k)^{-\tau} \ \prod_{\mathclap{\{i,j\}\in \Ecal_H}} \ \frac{t_it_j}{1+t_it_j} \ \ \prod_{\mathclap{\{u,w\}\notin \Ecal_H}} \  \frac{1}{1+t_ut_w}\dd t_1\cdots \dd t_k\\
			& = \bigO{n^{\frac{k}{2}(3-\tau)}}h_1(\varepsilon),
		\end{aligned}
	\end{equation}
	where $h_1(\varepsilon)$ is a function of $\varepsilon$. By Lemma~\ref{lem:convNH}\ref{lem:Afinite}, $h_1(\varepsilon)\to 0$ as $\varepsilon\searrow 0$. 
	
	For the second term in~\eqref{eq:expdeg1small2}, we obtain
	\begin{align}
		&o\Big(n^k (\varepsilon \sqrt{\mu n})^{1-\tau} \int_{1}^{\infty}\cdots\int_{1}^{\infty}(x_2\cdots x_k)^{-\tau} g_n(\varepsilon \sqrt{\mu n},x_2,\dots,x_n) \dd x_2\dots \dd x_k \Big) \nonumber\\
		&= o\Big(n^k (\mu n)^{\frac{k}{2}(1-\tau)}\varepsilon^{1-\tau} \int_{0}^{\infty}\cdots\int_{0}^{\infty}(t_2\cdots t_k)^{-\tau} h(\varepsilon ,t_2,\dots,t_n) \dd t_2\dots \dd t_k\Big)\nonumber\\
		& = o\left(n^{\frac{k}{2}(3-\tau)}\right)h_2(\varepsilon),
	\end{align}
	where
		\begin{equation}
		h(t_1,\dots,t_k)=\prod_{\mathclap{\{i,j\}\in \Ecal_H}} \ \frac{t_it_j}{1+t_it_j} \ \ \prod_{\mathclap{\{u,w\}\notin \Ecal_H}}  \ \frac{1}{1+t_ut_w},
	\end{equation}
and $h_2(\varepsilon)$ is a function of $\varepsilon$.

	We can bound the situation where another vertex has degree smaller than $\varepsilon\sqrt{n}$, or where one of the vertices has degree larger than $\sqrt{n}/\varepsilon$, similarly. This yields
	\begin{equation}
		\Exp{N(H,\bar{W}_n^k(\varepsilon))} = \bigO{n^{\frac{k}{2}(3-\tau)}}h(\varepsilon) + o\Big(n^{\frac{k}{2}(3-\tau)}\Big)\tilde{h}(\varepsilon) ,
	\end{equation}
	for some function $h(\varepsilon)$ not depending on $n$ such that $h(\varepsilon)\to 0$ when $\varepsilon\searrow 0$ and some function $\tilde{h}(\varepsilon)$ not depending on $n$. By the Markov inequality, 
	\begin{equation}
		\begin{aligned}[b]
			N(H,\bar{W}_n^k(\varepsilon))=h(\varepsilon)\bigOp{n^{\frac{k}{2}(3-\tau)}}.
		\end{aligned}
	\end{equation}
	Thus, for any $\delta>0$,
	\begin{equation}
		\limsup_{\varepsilon\to 0}\limsup_{n\to\infty} \Prob{\frac{N(H,\bar{W}_n^k(\varepsilon))}{n^{k(3-\tau)/2}}>\delta }=0.
	\end{equation}
	Combining this with Lemma~\ref{lem:convNH}\ref{lem:convNH1} gives
	\begin{align}
		\frac{N(H)}{n^{\frac{k}{2}(3-\tau)}}\plim & c^k\mu^{-\frac{k}{2}(\tau-1)}\! \int_{0}^{\infty}\! \cdots\!  \int_{0}^{\infty}(x_1,\cdots x_k)^{-\tau}\prod_{\mathclap{\{i,j\}\in \Ecal_{H}}} \ \frac{x_ix_j}{1+x_ix_j} \nonumber\\
		& \quad \times \prod_{\mathclap{\{u,w\}\notin \Ecal_{H}}} \ \frac{1}{1+x_ux_w}		\dd x_1\cdots \dd x_k.
	\end{align}
\end{proof}


\subsection{Conditional expectation}
We will prove Lemma~\ref{lem:convNH} using a second moment method. Thus, we will first investigate the expected number of copies of induced subgraph $H$ in $\URGnd$, and then bound its variance.
Let $H$ be a subgraph on $k$ vertices, labeled as ${[k]}$, and $m$ edges, denoted by $e_1={\{i_1,j_1\},\ldots,e_m=\{i_m,j_m\}}$.


%

\begin{lemma}[Convergence of conditional expectation of $\sqrt{n}$ subgraphs]\label{lem:convsub}
	Let $H$ be a subgraph such that~\eqref{eq:maxeqsub} has a unique maximizer, and the maximum is attained at 0. Then,
	\begin{align}
	\frac{\Exp{N (H,W_n^k(\varepsilon))}}{n^{\frac{k}{2}(3-\tau)}}&  \to (C(\tau-1))^k\mu^{-\frac{k}{2}(\tau-1)}\int_{\varepsilon}^{1/\varepsilon}\!\!\cdots \int_{\varepsilon}^{1/\varepsilon}(x_1\cdots x_k)^{-\tau}\nonumber\\
	&  \times \prod_{\mathclap{\{i,j\}\in \Ecal_{H}}} \ \frac{x_ix_j}{1+x_ix_j} \ \ \ \ \prod_{\mathclap{\{u,v\}\notin \Ecal_{H}}} \ \frac{1}{1+x_ux_v}\dd x_1\cdots \dd x_k .
	\end{align}
\end{lemma}

\begin{proof} 
	We denote
	\begin{equation}
		h(d_1,\dots,d_k)=\prod_{{\{i,j\}\in \Ecal_{H}}}\frac{d_id_j}{L_n+d_id_j}\prod_{{\{u,v\}\notin \Ecal_{H}}}\frac{1}{L_n+d_ud_v}.
	\end{equation}
	As
	\begin{equation}
		\Exp{N (H,W_n^k(\varepsilon))}=\sum_{(v_1,\ldots,v_k)\in W_n^k(\varepsilon)}\Prob{\URGnd|_{\boldsymbol{v}}=\mathcal{E}_H},
	\end{equation}
	and $d_{v_i}\geq \varepsilon \sqrt{n}$ for $i\in[k]$, we get from Lemma~\ref{lem:psuburg}
	\begin{align}\label{eq:condex}
	&	\Exp{N (H,W_n^k(\varepsilon))}=\sum_{(v_1,\ldots,v_k)\in W_n^k(\varepsilon)}\prod_{\{i,j\}\in \Ecal_H}\frac{d_{v_i}d_{v_j}}{L_n+d_{v_i}d_{v_j}} \ \ \ \ \prod_{\mathclap{\{s,t\}\notin \Ecal_H}} \ \frac{1}{L_n+d_{v_s}d_{v_t}}(1+o(1))\nonumber\\
		& =(1+o(1))\sum_{1\leq i_1<i_2<\dots<i_k\leq n}h(d_{i_1},\dots,d_{i_k})\ind{i_1,i_2,\dots,i_k\in W_n^k(\varepsilon)}.
	\end{align}
		We then define the measure
	\begin{equation}
		\Mn([a,b])=\mu^{(\tau-1)/2}n^{(\tau-3)/2}\sum_{i\in[n]}\ind{d_i\in[a,b]\sqrt{\mu n}}.
	\end{equation}
	By~\cite[Eq. (4.19)]{gao2018}
	\begin{equation}\label{eq:measconv}
		\Mn([a,b])\to C(\tau-1)\int_{a}^{b}t^{-\tau}\dd t =:\lambda([a,b]).
	\end{equation}
	Then,
	\begin{align}
		& \frac{\sum_{1\leq i_1<i_2<\dots<i_k\leq n}h(d_{i_1},\dots,d_{i_k})\ind{i_1,i_2,\dots,i_k\in W_n^k(\varepsilon)}}{n^{\frac{k}{2}(3-\tau)}\mu^{-\frac{k}{2}(\tau-1)}}\nonumber\\
		& =\frac{1}{k!}\int_{\varepsilon}^{1/\varepsilon}\dots\int_{\varepsilon}^{1/\varepsilon} h(t_1,\dots,t_k)\dd \Mn(t_1)\dots\dd \Mn(t_k).
	\end{align}
	Because the function $h(t_1,\dots,t_k)$ is a bounded, continuous function on $[\varepsilon,1/\varepsilon]^k$, 
	\begin{equation}\label{eq:gconv}
		\begin{aligned}[b]
			& \frac{\sum_{1\leq i_1<i_2<\dots<i_k\leq n}h(d_1,\dots,d_k)\ind{i_1,i_2,\dots,i_k\in W_n^k(\varepsilon)}}{n^{\frac{k}{2}(3-\tau)}\mu^{-\frac{k}{2}(\tau-1)}}\\
			&  \to \frac{1}{k!}\int_{\varepsilon}^{1/\varepsilon}\dots \int_{\varepsilon}^{1/\varepsilon} h(t_1,\dots,t_k)\dd \lambda(t_1)\dots\dd \lambda(t_k)\\
			&=\frac{(C(\tau-1))^3}{k!}\int_{\varepsilon}^{1/\varepsilon}\dots\int_{\varepsilon}^{1/\varepsilon}(x_1\cdots x_k)^{-\tau}\prod_{\mathclap{\{i,j\}\in \Ecal_{H}}}\frac{x_ix_j}{1+x_ix_j} \ \ \ \ \prod_{\mathclap{\{u,v\}\notin \Ecal_{H}}}\frac{1}{1+x_ux_v}\dd x_1 \dots \dd x_k.
		\end{aligned}
	\end{equation}
	Combining this with~\eqref{eq:condex} proves the lemma.
\end{proof}

\subsection{Variance of the number of induced subgraphs}
We now study the variance of the number of induced subgraphs. The following lemma shows that the variance of the number of subgraphs is small compared to its expectation:
\begin{lemma}[Conditional variance for subgraphs]\label{lem:varsub}
	Let $H$ be a subgraph such that~\eqref{eq:maxeqsub} has a unique maximum attained at 0. Then, 
	\begin{equation}
	\frac{\Var{N (H,W_n^k(\varepsilon))}}{\Exp{N (H,W_n^k(\varepsilon))}^2}\to 0.
	\end{equation}
\end{lemma}
\begin{proof}
	By Lemma~\ref{lem:convsub}, 
	\begin{equation}
	\Exp{N (H,W_n^k(\varepsilon))}^2=\Theta(n^{(3-\tau)k}),
	\end{equation}
	Thus, we need to prove that the variance is small compared to $n^{(3-\tau)k}$. Denote $\boldsymbol{v}=(v_1,\ldots,v_k)$ and ${\boldsymbol{u}}=(u_1,\ldots,u_k)$ and, for ease of notation, we denote $G=\URGnd$. 
	We write the variance as
	\begin{align}\label{eq:varmotif}
	\Var{N (H,W_n^k(\varepsilon))}&= \sum_{\boldsymbol{v}\in W_n^k(\varepsilon)}\sum_{\boldsymbol{u}\in W_n^k(\varepsilon)}
	\Big(\Prob{G|_{\boldsymbol{v}}= \Ecal_H,G|_{\boldsymbol{u}}= \Ecal_H}\nonumber\\
	& \quad \quad -\Prob{G|_{\boldsymbol{v}}=\Ecal_H}\Probn{G|_{\boldsymbol{u}}= \Ecal_H}\Big).
	\end{align}
	This splits into various cases, depending on the overlap of $\boldsymbol{v}$ and $\boldsymbol{u}$. When $\boldsymbol{v}$ and $\boldsymbol{u}$ do not overlap, 
	\begin{equation}\label{eq:varbound}
	\begin{aligned}[b]
	&\sum_{\boldsymbol{v}\in W_n^k(\varepsilon)}\sum_{\boldsymbol{u}\in W_n^k(\varepsilon)}\big(\Prob{G|_{\boldsymbol{v}}= \Ecal_H,G|_{\boldsymbol{u}}= \Ecal_H} -\Prob{G|_{\boldsymbol{v}}= \Ecal_H}\Prob{G|_{\boldsymbol{u}}= \Ecal_H}\big)\\
	& = \sum_{\boldsymbol{v}\in W_n^k(\varepsilon)}\sum_{\boldsymbol{u}\in W_n^k(\varepsilon)}\big(\Prob{G|_{\boldsymbol{v}}= \Ecal_H}\Prob{G|_{\boldsymbol{u}}=\Ecal_H}(1+o(1)) \nonumber\\
	& \quad\quad\quad -\Prob{G|_{\boldsymbol{v}}= \Ecal_H}\Prob{G|_{\boldsymbol{u}}= \Ecal_H}\big)\\
	& = \Exp{N (H,W_n^k(\varepsilon))}^2o(1),
	\end{aligned}
	\end{equation}
by Lemma~\ref{lem:psuburg}.
The other contributions are when $\boldsymbol{v}$ and $\boldsymbol{u}$ overlap. In this situation, we bound the probability that induced subgraph $H$ is present on a specified set of vertices by 1.
When $\boldsymbol{v}$ and $\boldsymbol{u}$ overlap on $s\geq 1$ vertices, we bound the contribution to~\eqref{eq:varmotif} as
	\begin{equation}\label{eq:varsqrt}
	\begin{aligned}[b]
	\sum_{\mathclap{\boldsymbol{v},\boldsymbol{u}\in W_n^k(\varepsilon)\colon \abs{\boldsymbol{v}\cup\boldsymbol{u}}=2k-s}} \ \ \Prob{G|_{\boldsymbol{v}}=\Ecal_H,G|_{\boldsymbol{u}}= \Ecal_H}& \leq\abs{ \{i\colon d_i\in \sqrt{\mu n}[\varepsilon,1/\varepsilon]\}}^{2k-s}\\
	& =\bigO{n^\frac{(3-\tau)(2k-s)}{2}},
	\end{aligned}
	\end{equation}
by Assumption~\ref{ass:degreeall}. This is $o(n^{(3-\tau)k})$ for $\tau\in(2,3)$, as required. 
\end{proof}

\begin{proof}[Proof of Lemma~\ref{lem:convNH}]
	We start by proving part (i). By Lemma~\ref{lem:varsub} and Chebyshev's inequality, 
	\begin{equation}
		N (H,W_n^k(\varepsilon))=\Exp{N (H,W_n^k(\varepsilon)}(1+\op(1)).
	\end{equation}
	Combining this with Lemma~\ref{lem:convsub} proves Lemma~\ref{lem:convNH}(i). 
	Lemma~\ref{lem:convNH}(ii) is follows from Lemma~\ref{lem:S3intind} in the next section, when $|S_3^*|=k$. 
\end{proof}

\section{Major contribution to general subgraphs: proof of Theorem~\ref{thm:motifs}(i)}
\label{sec:proofsec2}
We first introduce some further notation. As before, we denote the degree of a vertex $i$ inside its subgraph $H$ by $d^{\sss{(H)}}_{i}$. Furthermore, for any $W\subseteq V_H$, we denote by $d^{\sss{(H)}}_{i,W}$ the number of edges from vertex $i$ to vertices in $W$.
Let $H$ be a connected subgraph, such that the optimum of~\eqref{eq:maxeqsub} is unique, and let ${\mathcal{P}}=(S_1^*,S_2^*,S_3^*)$ be the optimal partition. Define
	\begin{equation}\label{eq:zeta}
	\zeta_i=
	\begin{cases}
	1 & \text{if }d^{\sss{(H)}}_i=1,\\
	d^{\sss{(H)}}_{i,S_1^*}+d^{\sss{(H)}}_{i,S_3^*}+d^{\sss{(H)}}_{i,V_1} & \text{if }i\in S_1^*,\\
	d^{\sss{(H)}}_{i,V_1} +d^{\sss{(H)}}_{i,S_1^*}+d^{\sss{(H)}}_{i,S_2^*}-|S_2^*|-|S_3^*|+1& \text{if }i\in S_2^*,\\
	d^{\sss{(H)}}_{i,S_1^*}+d^{\sss{(H)}}_{i,V_1} +d^{\sss{(H)}}_{i,S_2^*}-|S_2^*|& \text{if }i\in S_3^*.
	\end{cases}
	\end{equation}

We now provide two lemmas that show that two integrals related to the solution of the optimization problem~\eqref{eq:maxeqsub} are finite. These integrals are the key ingredient in proving Theorem~\ref{thm:motifs}(i).
\begin{lemma}[{Induced subgraph integrals over $S_3^*$}]
	\label{lem:S3intind} 
	Suppose that the maximum in~\eqref{eq:maxeqalphaind} is uniquely attained by ${\mathcal{P}}=(S_1^*,S_2^*,S_3^*)$ with $|S_3^*|=s>0$, and say that $S_3^*=[s]$.
	Then
	\begin{equation}
		\label{eq:S3intind}
		\int_{0}^{\infty}\cdots \int_{0}^\infty \prod_{i \in [s]}x_i^{-\tau+\zeta_i}\prod_{\mathclap{\{i,j\}\in \Ecal_{S_3^*}}}\frac{x_ix_j}{1+x_ix_j} \ \ \prod_{\mathclap{\{u,w\}\notin \Ecal_{S_3^*}}}\frac{1}{1+x_ux_w}\dd x_s\cdots\dd x_1<\infty.
	\end{equation} 
\end{lemma}


\begin{lemma}[{Induced subgraph integrals over $S_1^*\cup S_2^*$}]
	\label{lem:S1S2intind}
	Suppose the optimal solution to~\eqref{eq:maxeqalphaind} is unique, and attained by ${\mathcal{P}}=(S_1^*,S_2^*,S_3^*)$. Say that $S_2^*=[t_2]$ and $S_1^*=[t_2+t_1]\setminus [t_2]$. Then, for every $a>0$,
	\begin{equation}\label{eq:S1S2intind}
		\begin{aligned}[b]
			\int_{0}^{a}\cdots \int_0^a\int_0^\infty\cdots\int_0^\infty& \prod_{\mathclap{j\in[t_1+t_2]}}x_j^{-\tau+\zeta_j} \ \prod_{\mathclap{\{i,j\}\in \Ecal_{S_1^*,S_2^*}}}\frac{x_ix_j}{1+x_ix_j} \\
			&  \times \prod_{\mathclap{\{i,j\}\notin \Ecal_{S_1^*,S_2^*}}}\frac{1}{1+x_ix_j}\dd x_{t_1+t_2}\cdots \dd x_1<\infty .
		\end{aligned}
	\end{equation}
\end{lemma}

The proofs of Lemma~\ref{lem:S3intind} and~\ref{lem:S1S2intind} are similar to the proofs of~\cite[Lemmas 7.2 and 7.3]{hofstad2017d} and are therefore deferred to Appendix~\ref{app:lemmas}. 

\begin{proof}[Proof of Theorem~\textup{\ref{thm:motifs}(i)}]
Note that $d_{\max}\leq M n^{1/(\tau-1)})$ by Assumption~\ref{ass:degreeall}. Define 
\begin{equation}
	\gamma_i^u(n)=\begin{cases}
		Mn^{1/(\tau-1)}& \text{if }i\in S_2^*,\\
		n^{\alpha_i}/\varepsilon_n & \text{else,}
	\end{cases}
\end{equation}
with $\alpha_i$ as in~\eqref{eq:alphasub}, and denote
\begin{equation}
	\gamma_i^l(n)=\begin{cases}
		1& \text{if }i\in V_1,\\
		\varepsilon_n n^{\alpha_i}& \text{else.}
	\end{cases}
\end{equation}
We then show that the expected number of subgraphs where the degree of at least one vertex $i$ satisfies $d_i\notin[\gamma_i^l(n),\gamma_i^u(n)]$ is small, similarly to the proof of Theorem~\ref{thm:sqrtsub} in Section \ref{sec:proof2}.

We first study the expected number of copies of $H$ where the first vertex has degree $d_{v_1}\in[1,\gamma_1^l(n))$ and all other vertices satisfy $d_{v_i}\in[\gamma_i^l(n),\gamma_i^u(n)]$, by integrating the probability that induced subgraph $H$ is formed over the range where vertex $v_1$ has degree $d_{v_1}\in[1,\gamma_1^l(n))$ and all other vertices satisfy $d_{v_i}\in[\gamma_i^l(n),\gamma_i^u(n)]$. Using Lemma~\ref{lem:psuburg}, and that the degree distribution can be bounded as $\Prob{D=k}\leq M_2k^{-\tau}$ for some $M_2>0$ by Assumption~\ref{ass:degreeall}, we bound the expected number of such copies of $H$ by
\begin{equation}
	\label{eq:Exp1small}
	\begin{aligned}[b]
		&\sum_{\boldsymbol{v}}\Exp{I(H, \boldsymbol{v})\ind{d_{v_1}<\gamma^l_1(n),d_{v_i}\in [\gamma_i^l(n),\gamma_i^u(n)] \ \forall i>1}}\\
		& \leq Kn^k\int_{1}^{\gamma_1^l(n)}\int_{\gamma_2^l(n)}^{\gamma_2^u(n)}\cdots \int_{\gamma_k^l(n)}^{\gamma_k^u(n)}(x_1\cdots x_k)^{-\tau}
		\prod_{\mathclap{\{i,j\}\in \Ecal_H}} \ \ \frac{x_ix_j}{L_n+x_ix_j} \ \ \ \prod_{\mathclap{\{u,w\}\notin \Ecal_H}} \ \frac{L_n}{L_n+x_ux_w}\dd x_k\cdots\dd x_1,
	\end{aligned}
\end{equation}
for some $K>0$, and where we recall that $I(H, \boldsymbol{v})=\ind{\URGnd|_{\boldsymbol{v}}=\Ecal_H}$.  This integral equals zero when vertex 1 is in $V_1$, since then $[1,\gamma_1^l(n))=\varnothing$. 
Suppose {that} vertex 1 is in $S_2^*$. W.l.o.g.\ assume that $S_2^*={[t_2]}$, $S_1^*={[t_1+t_2]\setminus [t_2]}$ and $S_3^*={[t_1+t_2+t_3]\setminus [t_1+t_2]}$. 
We bound $x_ix_j/(L_n+x_ix_j)$ by 
\begin{itemize}
	\item[(a)] $x_ix_j/L_n$ for $i,j\in S_1^*$;
	\item[(b)] $x_ix_j/L_n$ for $i$ or $j$ in $V_1$; 
	\item[(c)] $x_ix_j/L_n$ for $i\in S_1^*$, $j\in S_3^*$ or vice versa; and
	\item[(d)] 1 for $i,j\in S_2^*$ and $i\in S_2^*$, $j\in S_3^*$ or vice versa.
\end{itemize}
Similarly, we bound $L_n/(L_n+x_ix_j)$ by 
\begin{itemize}
	\item[(a)] 1 for $i,j\in S_1^*$;
	\item[(b)] 1 for $i$ or $j$ in $V_1$; 
	\item[(c)] 1 for $i\in S_1^*$, $j\in S_3^*$ or vice versa; and
	\item[(d)] $L_n/(x_ix_j)$ for $i,j\in S_2^*$ and $i\in S_2^*$, $j\in S_3^*$ or vice versa.
\end{itemize}

Combining these bounds with the change of variables $y_i=x_i/n^{\alpha_i}$ yields for~\eqref{eq:Exp1small}, for some $\tilde{K}>0$,
in the bound

	\begin{align}\label{eq:expnhsmall}
		&{\sum_{\boldsymbol{v}}\Exp{I(H, \boldsymbol{v})\ind{d_{v_1}<\gamma^l_1(n),d_{v_i}\in [\gamma_i^l(n),\gamma_i^u(n)] \ \forall i>1}}}\nonumber\\
		& \leq \tilde{K}n^k n^{|S_1^*|(2-\tau)+|S_3^*|(1-\tau)/2-|S_2^*|}n^{\frac{\tau-3}{\tau-1}E_{S_1^*}+\frac{\tau-3}{2(\tau-1)}E_{S_1^*,S_3^*}-\frac{1}{\tau-1}E_{S_1^*,V_1}-\frac{1}{2}E_{S_3^*,V_1}-\frac{\tau-2}{\tau-1}E_{S_2^*,V_1}}\nonumber\\
		& \quad \times n^{\left(\frac 12 |S_2|(|S_2|-1)-E_{S_2}\right)\frac{\tau-3}{\tau-1}+(|S_2||S_3|-E_{S_2,S_3})\frac{\tau-3}{2(\tau-1)}}\nonumber\\
		& \quad \times  \int_{0}^{\varepsilon_n}\int_{0}^{M}\cdots\int_{0}^{M}\int_{0}^{\infty}\cdots \int_{0}^{\infty}\prod_{i\in V_H\setminus V_1}y_i^{-\tau+\zeta_i}\prod_{{\{i,j\}\in \Ecal_{S_3^*}\cup \Ecal_{S_1^*,S_2^*}}}\frac{y_iy_j}{y_iy_j+1}\nonumber\\
		& \quad  \times\prod_{{\{u,w\}\notin \Ecal_{S_3^*}\cup \Ecal_{S_1^*,S_2^*}}}\frac{1}{y_uy_w+1}\dd y_{t_1+t_2+t_3}\cdots \dd y_{1} \prod_{j \in V_1}\int_{1}^{\infty} y_j^{1-\tau}\dd y_j,
	\end{align}
where the integrals from 0 to $M$ correspond to vertices in $S_2^*$ and the integrals from 0 to $\infty$ to vertices in $S_1^*$ and $S_3^*$. Since $\tau\in(2,3)$, the integrals corresponding to vertices in $V_1$ are finite. By the analysis from~\eqref{eq:maxtemp} to~\eqref{eq:maxcontrscaling}, 
\begin{align}\label{eq:StoB}
	&|S_1^*|(2-\tau)+|S_3^*|(1-\tau)/2-|S_2^*|+k+\frac{\tau-3}{\tau-1}E_{S_1^*}+\frac{\tau-3}{2(\tau-1)}E_{S_1^*,S_3^*}\nonumber\\
	&\qquad\qquad -\frac{1}{\tau-1}E_{S_1^*,V_1}
	-\frac{1}{2}E_{S_3^*,V_1}-\frac{\tau-2}{\tau-1}E_{S_2^*,V_1}\nonumber\\
	&\qquad\qquad +\left(\frac 12 |S_2|(|S_2|-1)-E_{S_2}\right)\frac{\tau-3}{\tau-1}+(|S_2||S_3|-E_{S_2,S_3})\frac{\tau-3}{2(\tau-1)}\nonumber \\
	&\qquad= \frac{3-\tau}{2}(k_{2+}+B(H))+k_1/2.
\end{align}
The integrals over $y_i\in V_H\setminus V_1$ can be split into
\begin{equation}
	\label{eq:ints2s3}
	\begin{aligned}[b]
		& \int_{0}^{\varepsilon_n}\int_{0}^{M}\cdots\int_{0}^{M}\int_{0}^{\infty}\cdots \int_{0}^{\infty} \ \prod_{\mathclap{i\in S_1^*\cup S_2^*}} \ y_i^{-\tau+\zeta_i} \ \ \prod_{\mathclap{\{i,j\}\in \Ecal_{S_1^*,S_2^*}}}\frac{y_iy_j}{y_iy_j+1} \ \ \prod_{\mathclap{\{u,w\}\notin \Ecal_{S_1^*,S_2^*}}}\frac{1}{y_uy_w+1}\dd y_{t_1+t_2}\cdots \dd y_{1}\\
		& \quad \times\int_{0}^{\infty}\cdots \int_{0}^{\infty}\prod_{i\in S_3^*}y_i^{-\tau+\zeta_i} \ \ \prod_{\mathclap{\{i,j\}\in \Ecal_{S_3^*}}} \ \frac{y_iy_j}{y_iy_j+1} \ \ \prod_{\mathclap{\{u,w\}\notin \Ecal_{S_3^*}}} \ \ \frac{1}{y_uy_w+1}\dd y_{t_1+t_2+t_3}\cdots \dd y_{t_1+t_2+1}.
	\end{aligned}
\end{equation}
By Lemma~\ref{lem:S3intind} the set of integrals on the second line of~\eqref{eq:ints2s3} is finite. Lemma~\ref{lem:S1S2intind} shows that the set of integrals on the first line of~\eqref{eq:ints2s3} tends to zero for $\varepsilon_n\to 0$. Thus, 
\begin{align}\label{eq:inteps}
	\int_{0}^{\varepsilon_n}\!&\int_{0}^{M}\!\!\cdots\int_{0}^{M}\!\!\int_{0}^{\infty}\!\!\cdots \!\int_{0}^{\infty} \ \prod_{\mathclap{i\in S_1^*\cup S_2^*}} \ y_i^{-\tau+\zeta_i} \ \prod_{\mathclap{\{i,j\}\in \Ecal_{S_1^*,S_2^*}}}  \ \ \frac{y_iy_j}{y_iy_j+1} \ \prod_{\mathclap{\{u,w\}\notin \Ecal_{S_1^*,S_2^*}}} \ \ \frac{1}{y_uy_w+1}\dd y_{t_1+t_2}\cdots \dd y_{1}\nonumber\\
	& = o(1).
\end{align}
Therefore,~\eqref{eq:expnhsmall},\eqref{eq:StoB} and~\eqref{eq:inteps} yield
\begin{align}\label{eq:NHo}
	&{\sum_{\boldsymbol{v}}\Exp{I(H, \boldsymbol{v})\ind{d_{v_1}<\gamma^l_1(n),d_{v_i}\in [\gamma_i^l(n),\gamma_i^u(n)] \ \forall i>1}}}\nonumber\\
	&\qquad=o\left(n^{\frac{3-\tau}{2}(k_{2+}+B(H))+k_1/2}\right),
\end{align}
when vertex 1 is in $ S_2^*$. Similarly, we can show that the expected contribution from $d_{v_1}<\gamma_1^l(n)$ satisfies the same bound when vertex 1 is in $S_1^*$ or $S_3^*$. The expected number of subgraphs where $d_{v_1}>\gamma_1^u(n)$ if vertex 1 is in $S_1^*$, $S_3^*$ or $V_1$ can be bounded similarly, as well as the expected contribution where multiple vertices have $d_{v_i}\notin [\gamma_i^l(n),\gamma_i^u(n)]$. 

Denote
\begin{equation}
	\Gamma_n(\varepsilon_n) = \{(v_1,\dots,v_k)\colon d_{v_i}\in[\gamma_{v_i}^l(n),\gamma_{v_i}^u(n)] \},
\end{equation}	
and define $\bar{\Gamma}_n(\varepsilon_n)$ as its complement. Denote the number of subgraphs with vertices in $\bar{\Gamma}_n(\varepsilon_n)$ by $N(H,\bar{\Gamma}_n(\varepsilon_n))$. Since $d_{\max}\leq Mn^{1/(\tau-1)}$, $\Gamma_n(\varepsilon_n)={M}_n^{(\boldsymbol{\alpha})}$. Therefore, 
\begin{equation}
	N\Big(H,\bar{M}_n^{(\boldsymbol{\alpha})}\left(\varepsilon_n\right)\Big) = N\Big(H,\bar{\Gamma}_n(\varepsilon_n)\Big),
\end{equation}
where $N\Big(H,\bar{M}_n^{(\boldsymbol{\alpha}))}\left(\varepsilon_n\right)\big)$ denotes the number of copies of $H$ on vertices not in $M_n^{(\boldsymbol{\alpha})}\left(\varepsilon_n\right)$. 
By the Markov inequality and~\eqref{eq:NHo},
\begin{equation}
	N(H,\bar{M}_n^{(\boldsymbol{\alpha})}(\varepsilon))= N\Big(H,\bar{\Gamma}_n(\varepsilon_n)\Big)=o\left(n^{\frac{3-\tau}{2}(k_{2+}+B(H))+k_1/2}\right).
\end{equation}

Combining this with Theorem~\ref{thm:motifs}(ii), for fixed $\varepsilon>0$, 
\begin{align} 
	N(H)&= N(H,M_n^{(\boldsymbol{\alpha})}(\varepsilon))+N(H,\bar{M}_n^{(\boldsymbol{\alpha})}(\varepsilon))=O(n^{\frac{3-\tau}{2}(k_{2+}+B(H))+k_1/2})
\end{align}
shows that
\begin{equation}
	N\Big(H,M_n^{(\boldsymbol{\alpha})}\left(\varepsilon_n\right)\big)/{N(H)}\plim 1,
\end{equation}
as required. This completes the proof of Theorem~\ref{thm:motifs}(i).
\end{proof}

\section{Proof of Theorem~\ref{thm:alg}}\label{sec:algproof}

\begin{algorithm}
	\caption{Finding induced subgraph $H$ of Figure~\ref{fig:Hgraph}.}\label{alg:urgecm}
	\SetKwInOut{Input}{Input}\SetKwInOut{Output}{Output}
	\Input{$G=(V,E)$.}
	\Output{Location of $H$ in $G$ or fail.}
	Define $n=\abs{V}$, $\varepsilon_n=1/\log(n)$, $I_n=[ n^{1/(\tau-1)}\varepsilon_n, n^{1/(\tau-1)}/\varepsilon_n]$ $J_n=[ n^{(\tau-2)/(\tau-1)}\varepsilon_n, n^{(\tau-2)/(\tau-1)}/\varepsilon_n]$ and set $V'=\emptyset$ and $W'=\emptyset$.\\
	\For{$i\in V$}{
		\lIf{ $d_i\in I_n$}{ $V'=V'\cup i$}
	\lIf{ $d_i\in J_n$}{ $W'=W'\cup i$}}
	Divide the vertices in $V'$ randomly into $\lfloor \abs{V'}/2\rfloor$ pairs $S_1,\dots,S_{\lfloor\abs{V'}/k\rfloor}$.\\
	Divide the vertices in $W'$ randomly into $\lfloor \abs{W'}/4\rfloor$ sets of size 4, $T_1,\dots,T_{\lfloor\abs{V'}/k\rfloor}$.\\
	Set $k=0$\\
	\For{$j=1,\dots,\lfloor\abs{V'}/2\rfloor$}{
		\For {$i=1,\dots,\abs{W'}/4\rfloor$}{
			$k=k+1$.\\
		\lIf{ $H$ is an induced subgraph on $S_j\cup T_i$}{\Return location of $H$}
		\lIf{ $k=n$}{\Return \textbf{fail}.}
	}}
\Return \textbf{fail}
\end{algorithm}
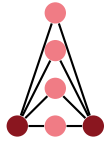
\begin{figure}[tb]
	\definecolor{mycolor1}{RGB}{230,37,52}%
	\tikzstyle{every node}=[circle,fill=black!85,minimum size=8pt,inner sep=0pt]
	\tikzstyle{S1}=[fill=mycolor1!60]
	\tikzstyle{S2}=[fill=mycolor1!60!black]
	\tikzstyle{S3}=[fill=mycolor1]
	\tikzstyle{n1}=[fill=mycolor1!20]
		\centering
		\begin{tikzpicture}
			\tikzstyle{edge} = [draw,thick,-]
			\node[S2] (a) at (0,0) {};
			\node[S2] (c) at (1,0) {};
			\node[S1] (b) at (0.5,0) {};
			\node[S1] (d) at (0.5,1) {};
			\node[S1] (e) at (0.5,1.5) {};
			\node[S1] (f) at (0.5,0.5) {};
			\draw[edge] (a)--(b);
			\draw[edge] (c)--(b);
			\draw[edge] (a)--(d);
			\draw[edge] (c)--(d);
			\draw[edge] (a)--(f);
			\draw[edge] (c)--(f);
			\draw[edge] (c)--(e);
			\draw[edge] (a)--(e);
		\end{tikzpicture}	
	\caption{The subgraph $H$ that is used in Algorithm~\ref{alg:urgecm}. Algorithm~\ref{alg:urgecm} attempts to find a copy of $H$ where the dark vertices are in $V'$, and the light vertices in $W'$.}
	\label{fig:Hgraph}
\end{figure}

\begin{proof}[Proof of Theorem~\ref{thm:alg}]
	 Algorithm~\ref{alg:urgecm} shows the algorithm that distinguishes uniform random graphs from rank-1 inhomogeneous random graphs with connection probabilities~\eqref{eq:ecm} or~\eqref{eq:hvm}. It first selects only vertices of degrees proportional to $n^{1/(\tau-1)}$ and $n^{(\tau-2)/(\tau-1)}$, and then randomly selects such vertices and checks whether they form a copy of induced subgraph $H$ of Figure~\ref{fig:Hgraph}. We will show that with high probability, Algorithm~\ref{alg:urgecm} finds a copy of $H$ when the input graph is generated by a uniform random graph, and that with high probability, Algorithm~\ref{alg:urgecm} outputs `fail' when the input graph is a rank-inhomogeneous random graph with connection probability ~\eqref{eq:ecm} or~\eqref{eq:hvm}.
	
	We first focus on the performance of Algorithm~\ref{alg:urgecm} when the input $G$ is a uniform random graph. Algorithm~\ref{alg:urgecm} detects copies of subgraph $H$ where the vertices have degrees as illustrated in Figure~\ref{fig:subdifs}(c): two vertices of degree proportional to $n^{1/(\tau-1)}$ and four of degree proportional to $n^{(\tau-2)/(\tau-1)}$. 
	By Theorem~\ref{thm:motifs}(ii), there are at least $cn^{(4-\frac{1}{\tau-1})(3-\tau)}$ such induced subgraphs for some $c$ with high probability. Furthermore, denote \begin{equation}
		\boldsymbol{\alpha}=[n^{1/(\tau-1)},n^{1/(\tau-1)},n^{(\tau-2)/(\tau-1)},n^{(\tau-2)/(\tau-1)},n^{(\tau-2)/(\tau-1)},n^{(\tau-2)/(\tau-1)}].
	\end{equation} 
	By Assumption~\ref{ass:degrees},
	\begin{equation}\label{eq:Mnalg}
		|\Mn(\boldsymbol{\alpha})|=\Theta(n^{4(3-\tau)}\log(n)^{6(\tau-1)}),
	\end{equation}
 so that there are at most $c_2n^{4(3-\tau)}\log(n)^{6(\tau-1)}$ sets of vertices with degrees in $\Mn(\boldsymbol{\alpha})$ that form no copy of induced subgraph $H$ for some $c_2<\infty$. Thus, the probability that a randomly chosen set of vertices with degrees in $\Mn(\boldsymbol{\alpha})$ forms $H$ is at least 
	\begin{equation}
		\frac{cn^{(4-\frac{1}{\tau-1})(3-\tau)}}{c_2n^{4(3-\tau)}\log(n)^{6(\tau-1)}}=\tfrac{c}{c_2}n^{(\tau-3)/(\tau-1)}\log(n)^{6(1-\tau)}.
	\end{equation}
Algorithm~\ref{alg:urgecm} tries at most $n$ such sets of vertices with degrees in $\Mn(\boldsymbol{\alpha})$, and therefore attempts to find subgraph $H$ in $\Theta(\min(n,n^{4(3-\tau)}\log(n)^{6(\tau-1)}))=\Theta(f(n))$ attempts where $f(n)=\min(n,n^{4(3-\tau)}\log(n)^{6(\tau-1)})$. Thus, the probability that the algorithm does not find a copy of induced subgraph $H$ among all attempts is bounded by
\begin{equation}
	\Prob{\text{Algorithm does not find $H$}}\leq \left(1-n^{\frac{\tau-3}{\tau-1}}\right)^{f(n)}\leq \me^{-n^{\gamma}},
\end{equation}
for some $\gamma>0$, where we have used that $1-x\leq \me^{-x}$

We now analyze the performance of Algorithm~\ref{alg:urgecm} on rank-1 inhomogeneous random graphs with connection probability~\eqref{eq:ecm}. As these have the same degree distribution asymptotically,~\eqref{eq:Mnalg} also holds there. Furthermore, the probability that vertices in $\Mn(\boldsymbol{\alpha})$ together form a copy of $H$ is 
\begin{equation}
	\prod_{{\{i,j\}\in \Ecal_{H}}}p(i,j)\prod_{{\{i,j\}\notin \Ecal_{H}}}(1-p(i,j))\leq \me^{-n^{\frac{1}{\tau-1}}n^{\frac{1}{\tau-1}}/\log(n)^2}\leq\me^{-n^{\gamma_2}}
\end{equation}
for some $\gamma_2>0$,
where we bounded all $p(i,j)$ and $1-p(i,j)$ by 1, except for $1-p(i,j)$ for the non-edge between the two vertices of degree at least $n^{\frac{1}{\tau-1}}/\log(n)$ (vertices in the left and right bottom corner of Figure~\ref{fig:Hgraph}). Thus, there are at most $\Theta(n^{4(3-\tau)}\me^{-n^{\frac{3-\tau}{\tau-1}}}\log(n)^{6(\tau-1)})$ copies of induced subgraph $H$ on sets of vertices in $\Mn(\boldsymbol{\alpha})$. Therefore, the probability that  a randomly chosen set of vertices with degrees in $\Mn(\boldsymbol{\alpha})$ forms $H$ is at most
\begin{equation}
	c_3\frac{n^{4(3-\tau)}\me^{-n^{\frac{3-\tau}{\tau-1}}}\log(n)^{6(\tau-1)}}{n^{4(3-\tau)}\log(n)^{6(\tau-1)}}=c_3\me^{-n^{\frac{3-\tau}{\tau-1}}}
\end{equation}
Then, the probability that the algorithm does not find a copy of induced subgraph $H$ among all $n$ attempts is bounded by
\begin{align}
	\Prob{\text{Algorithm does not find $H$}}&\geq \left(1-c_3\me^{-n^{\frac{3-\tau}{\tau-1}}}\right)^{f(n)}\nonumber\\
	& =1-c_3f(n)\me^{-n^{\frac{3-\tau}{\tau-1}}}+\bigO{f(n)^2\me^{-2n^{\frac{3-\tau}{\tau-1}}}}.
\end{align}
Thus, with high probability the algorithm outputs `fail' when the input graph is a rank-1 inhomogeneous random graph with connection probability~\eqref{eq:ecm}. A similar calculation shows that the algorithm outputs `fail' with high probability when the input graph $G$ is a rank-1 inhomogeneous random graph with connection probabilities~\eqref{eq:hvm}.
\end{proof}
\bibliographystyle{abbrv}
\DeclareRobustCommand{\VAN}[3]{#3}

\bibliography{../references}

\appendix
\section{Proof of Lemma~\ref{lem:maxmotif}}\label{sec:prooflembeta}
\begin{proof}[Proof of Lemma~\ref{lem:maxmotif}]
	By defining $\beta_i=\alpha_i-\tfrac{1}{2}$ and
	\begin{equation}
		a_{ij}(\beta_i,\beta_j)=\begin{cases}
			1 & \beta_i+\beta_j<0,\{i,j\}\in\mathcal{E}_H,\\
			-1 & \beta_i+\beta_j>0,\{i,j\}\notin\mathcal{E}_H,\\
			0 & \text{else},
		\end{cases}
	\end{equation}
	we can rewrite~\eqref{eq:maxeqalphaind} as
	\begin{equation}\label{eq:maxeqbeta}
		\max_{\boldsymbol{\beta}} \frac{1-\tau}{2}k+ \sum_{i}\beta_i(1-\tau+\sum_{j\neq i}a_{ij}(\beta_i,\beta_j)),
	\end{equation} 
	over all possible values of $\beta_i\in[-\tfrac12,\tfrac{3-\tau}{2(\tau-1)}]$. We ignore the constant factor of $(1-\tau)\tfrac{k}{2} $ in~\eqref{eq:maxeqbeta}, since it does not influence the optimal $\beta$ values. Then, we have to prove that $\beta_i\in\{-\tfrac 12, \tfrac{\tau-3}{2(\tau-1)},0,\tfrac{3-\tau}{2(\tau-1)}\}$ for all $i$ in the optimal solution.
	Note that~\eqref{eq:maxeqbeta} is a piecewise linear function in $\beta_1,\dots,\beta_k$. Therefore, if~\eqref{eq:maxeqbeta} has a unique maximum, then it must be attained at the boundary for $\beta_i$ or at a border of one of the linear sections. Thus, any unique optimal value of $\beta_i$ satisfies $\beta_i=-\tfrac{1}{2}$, $\beta_i=\tfrac{\tau-3}{2(\tau-1)}$ or $\beta_i+\beta_j=0$ for some $j$.
	
	%
	
	The proof of the lemma then consists of three steps: \\
	\textit{Step 1.} Show that $\beta_i=-\tfrac{1}{2}$ if and only if vertex $i$ has degree 1 in $H$ in any optimal solution.\\
	\textit{Step 2.} Show that any unique solution does not contain $i$ with $\abs{\beta_i}\in(0,\tfrac{3-\tau}{2(\tau-1)})$.\\
	\textit{Step 3.} Show that any optimal solution that is not unique can be transformed into two different optimal solutions with $\beta_i\in\{-\tfrac 12, \tfrac{\tau-3}{2(\tau-1)},0,\tfrac{3-\tau}{2(\tau-1)}\}$ for all $i$.\\
	
	\textit{Step 1.}
	Let $i$ be a vertex of degree 1 in $H$, and $j$ be the neighbor of $i$. 
	The contribution from vertex $i$ to~\eqref{eq:maxeqbeta} is
	\begin{equation}
		\beta_i(1-\tau+\ind{\beta_i<-\beta_j}-\sum_{s\neq i,j}\ind{\beta_i>-\beta_s}).
	\end{equation} 
	This contribution is maximized when choosing $\beta_i=-\tfrac{1}{2}$ as $\tau\in(2,3)$. 
	Thus, $\beta_i=- \tfrac{1}{2}$ in the optimal solution if the degree of vertex $i$ is one.
	
	Let $i$ be a vertex in $V_H$, and recall that $d_i^{\sss{(H)}}$ denotes the degree of $i$ in $H$. Let $i$ be such that $d_i^{\sss{(H)}}\geq 2$ in $H$, and suppose that $\beta_i<\tfrac{\tau-3}{2(\tau-1)}$. Because the maximal value of $\beta_j$ for $j\neq i$ is $\tfrac{3-\tau}{2(\tau-1)}$. This implies that $\beta_i+\beta_j<0$ for all $j$. Thus, the contribution to the $i$th term of~\eqref{eq:maxeqbeta} is
	\begin{equation}
		-\tfrac{1}{2}(1-\tau+d_i^{\sss{(H)}})<0,
	\end{equation}
	for any $\beta_j$, $j\neq i$.
	Increasing $\beta_i$ to $\tfrac{\tau-3}{2(\tau-1)}$ then gives a higher contribution. 
	\cs{Thus, $\beta_i\geq \tfrac{\tau-3}{2(\tau-1)}$ when $d_i^{\sss{(H)}}\geq 2$.}

	\textit{Step 2.}
	Now we show that when the solution to~\eqref{eq:maxeqbeta} is unique, it is never optimal to have $\abs{\beta}\in(0,\tfrac{3-\tau}{2(\tau-1)})$. 
	Let 
	\begin{equation}\label{eq:tildebeta}
		\tilde{\beta}=\min_{i:\abs{\beta_i}>0}\abs{\beta_i}.
	\end{equation}
	Let  $N_{\tilde{\beta}^-}$ denote the number of vertices with their $\tilde{\beta}$ value equal to $-\tilde{\beta}$, and $N_{\tilde{\beta}^+}$ the number of vertices with value $\tilde{\beta}$, where $N_{\tilde{\beta}^+}+N_{\tilde{\beta}^-}\geq 1$. Furthermore, let $E_{\tilde{\beta}^-}^+$ denote the number of edges from vertices with value $-\tilde{\beta}$ to other vertices $j$ such that $\beta_j<\tilde{\beta}$, and $E_{\tilde{\beta}^+}^+$ the number of edges from vertices with value $\tilde{\beta}$ to other vertices $j$ such that $\beta_j<-\tilde{\beta}$. Similarly, let $E_{\tilde{\beta}^-}^-$ denote the number of non-edges from vertices with value $-\tilde{\beta}$ to other vertices $j$ such that $\beta_j>\tilde{\beta}$, and $E_{\tilde{\beta}^+}^-$ the number of non-edges from vertices with value $\tilde{\beta}$ to other vertices $j$ such that $\beta_j<-\tilde{\beta}$.  Then, the contribution from these vertices to~\eqref{eq:maxeqbeta} is
	\begin{equation}\label{eq:Nbeta}
		\tilde{\beta}\big((1-\tau)(N_{\tilde{\beta}^+}-N_{\tilde{\beta}^-})+E_{\tilde{\beta}^+}^+-E_{\tilde{\beta}^-}^+-E_{\tilde{\beta}^+}^-+E_{\tilde{\beta}^-}^-\big).
	\end{equation}
	Because we assume ${\beta}$ to be optimal, and the optimum to be unique, the value inside the brackets cannot equal zero. The contribution is linear in $\tilde{\beta}$ and it is the optimal contribution, and therefore $\tilde{\beta}\in\{0,\tfrac{3-\tau}{2(\tau-1)}\}$.
	This shows that $\beta_i\in\{\tfrac{\tau-3}{2(\tau-1)},0,\tfrac{3-\tau}{2(\tau-1)}\}$ for all $i$ such that $d_i^{\sss{(H)}}\geq 2$.
	
	\textit{Step 3.}
	\cs{
		Suppose that the solution to~\eqref{eq:maxeqbeta} is not unique. Suppose that $\beta_*$ appears in one of the optimizers of~\eqref{eq:maxeqbeta}. In the same notation as in~\eqref{eq:Nbeta}, the contribution from vertices with $\beta$-values $\beta_*$ and $-\beta_*$ equals
		\begin{equation}
			{\beta}_*\Big[(1-\tau)\big(N_{{\beta}_*^+}-N_{{\beta}_*^-}\big)+E_{{\beta}_*^+}^+-E_{{\beta}_*^-}^+-E_{{\beta}_*^+}^-+E_{{\beta}_*^-}^-\Big]. 
		\end{equation}
		Since this contribution is linear in $\beta_*$, the contribution of these vertices can only be non-unique if the term within the square brackets equals zero. 
		Thus, for the solution to~\eqref{eq:maxeqbeta} to be non-unique, there must exist $\hat{\beta}_1,\ldots,\hat{\beta}_s>0$ for some $s\geq 1$ such that  
		\begin{equation}
			\hat{\beta}_j\Big((1-\tau)\big(N_{\hat{\beta}_j^+}-N_{\hat{\beta}_j^-}\big)+E_{{\beta}_*^+}^+-E_{{\beta}_*^-}^+-E_{{\beta}_*^+}^-+E_{{\beta}_*^-}^-\Big)=0 \quad \forall j\in[s].
		\end{equation}
		Setting all $\hat{\beta}_j=0$ and setting all $\hat{\beta}_j=\tfrac{3-\tau}{2(\tau-1)}$ are both optimal solutions. Thus, if the solution to~\eqref{eq:maxeqbeta} is not unique, at least 2 solutions exist with $\beta_i\in\{\tfrac{\tau-3}{2(\tau-1)},0,\tfrac{3-\tau}{2(\tau-1)}\}$ for all $i\in V_H$. }
\end{proof}

\section{Proof of Lemmas~\ref{lem:S3intind} and~\ref{lem:S1S2intind}}\label{app:lemmas}

We first provide a lemma that states several properties of the variable $\zeta_i$ of~\eqref{eq:zeta}, that will appear often in the integrals we have to bound:
\begin{lemma}[Bounds on $\zeta_i$]
	\label{lem:dmotif}
	Let $H$ be a connected subgraph, such that the optimum of~\eqref{eq:maxeqsub} is unique, and let ${{\mathcal{P}}=(S_1^*,S_2^*,S_3^*)}$ be the optimal partition. Then
	\begin{enumerate}[label={\upshape(\roman*)}]
		\item $\zeta_i+d^{\sss{(H)}}_{i,S_2^*}-|S_2^*|\leq 1$ for $i\in S_1^*$;
		\item $d^{\sss{(H)}}_{i,S_3^*}+\zeta_i\geq 1$ for $i\in S_2^*$;
		\item $\zeta_i+d^{\sss{(H)}}_{i,S_3^*}-|S_3^*|\leq 0$ and $d^{\sss{(H)}}_{i,S_3^*}+\zeta_i\geq 2$ for $i\in S_3^*$.
	\end{enumerate}
\end{lemma}

\begin{proof} Suppose first that $i\in S_1^*$. Now consider the partition $\hat{S}_1=S_1^*\setminus \{i\}$, $\hat{S}_2=S_2^*$, $S_3=S_3^*\cup \{i\}$. Then, $E_{\hat{S}_1}=E_{S_1^*}-d^{\sss{(H)}}_{i,S_1^*}$, $E_{\hat{S}_1,\hat{S}_3}=E_{S_1^*,S_3^*}+d^{\sss{(H)}}_{i,S_1^*}-d^{\sss{(H)}}_{i,S_3^*}$ and $E_{\hat{S}_2,\hat{S}_3}=E_{S_2^*,S_3^*}+d^{\sss{(H)}}_{i,S_2^*}$. Furthermore, $E_{\hat{S}_1,V_1}=E_{S^*_1,V_1}-d^{\sss{(H)}}_{i,V_1}$ and $E_{\hat{S}_2,V_1}=E_{S_2^*,V_1}$.  Because the partition into $S_1^*,S_2^*$ and $S_3^*$ achieves the unique optimum of~\eqref{eq:maxeqsub},
	\begin{equation}
		\begin{aligned}[b]
			&|S_1^*|+|S_2^*|\frac{(2-\tau-k+|S_1^*|+k_1)}{\tau-1}-\frac{2E_{S_1^*}-2E_{S_2^*}+E_{S_1^*,S_3^*}-E_{S_2^*,S_3^*}+E_{S_1^*,V_1}-E_{S_2^*,V_1}}{\tau-1}\\
			&>|S_1^*|-1+|S_2^*|\frac{(1-\tau-k+|S_1^*|+k_1)}{\tau-1}\\
			& \quad -\frac{2E_{S_1^*}-2E_{S_2^*}+E_{S_1^*,S_3^*}-E_{S_2^*,S_3^*}-d^{\sss{(H)}}_{i,S_2^*}-d^{\sss{(H)}}_{i,S_1^*}-d^{\sss{(H)}}_{i,S_3^*}+E_{S_1^*,V_1}-E_{S_2^*,V_1}-d^{\sss{(H)}}_{i,V_1}}{\tau-1},
		\end{aligned}
	\end{equation}
	which reduces to
	\begin{equation}
		d^{\sss{(H)}}_{i,S_1^*}+d^{\sss{(H)}}_{i,S_3^*}+d^{\sss{(H)}}_{i,V_1}-d^{\sss{(H)}}_{i,S_2^*}-|S_2^*|=\zeta_i-d^{\sss{(H)}}_{i,S_2^*}-|S_2^*|<\tau-1.
	\end{equation}
	Using that $\tau\in(2,3)$ then yields $d^{\sss{(H)}}_{i,S_1^*}+d^{\sss{(H)}}_{i,S_3^*}+d^{\sss{(H)}}_{i,V_1}\leq 1$. 
	
	Similar arguments give the other inequalities. For example, for $i\in S_3^*$, considering the partition where $i$ is moved to $S_1^*$ gives the inequality $d^{\sss{(H)}}_{i,S_3^*}+d^{\sss{(H)}}_{i,S_1^*}+d^{\sss{(H)}}_{i,V_1}\geq 2$, and considering the partition where $i$ is moved to $S_2^*$ results in the inequality $d^{\sss{(H)}}_{i,S_1^*}+d^{\sss{(H)}}_{i,V_1}\leq 1$, so that $\zeta_i\leq 1$.
\end{proof}

\begin{proof}[Proof of Lemma \ref{lem:S3intind}] Recall that $S_3^*=[s]$. 
	First of all,
	\begin{align}\label{eq:S3intind2}
		&\int_{0}^{\infty}\cdots \int_{0}^\infty \prod_{i \in [s]}x_i^{-\tau+\zeta_i}\prod_{\mathclap{\{i,j\}\in \Ecal_{S_3^*}}}\frac{x_ix_j}{1+x_ix_j} \ \ \prod_{\mathclap{\{i,j\}\notin \Ecal_{S_3^*}}}\frac{1}{1+x_ix_j}\dd x_s\cdots\dd x_1\nonumber\\
		&\leq \int_{0}^{\infty}\cdots \int_{0}^\infty \prod_{i \in [s]}x_i^{-\tau+\zeta_i}\prod_{\mathclap{\{i,j\}\in \Ecal_{S_3^*}}}\min(x_ix_j,1)\prod_{\mathclap{\{i,j\}\notin \Ecal_{S_3^*}}}\min(1/(x_ix_j),1)\dd x_s\cdots\dd x_1.
	\end{align} 
	
	We compute the contribution to~\eqref{eq:S3intind2} where the integrand runs from 1 to $\infty$ for vertices in some nonempty set $U$, and from 0 to 1 for vertices in $\bar{U}=S_3^*\setminus{U}$. W.l.o.g., assume that $U={[t]}$ for some $1\leq t< s$ and that $x_1<x_2<\cdots <x_t$. Define, for $i\in \bar{U}$,
	\begin{align}
		&\tilde{h}(i,\boldsymbol{x})= \int_{0}^{1}x_i^{-\tau+\zeta_i+d^{\sss{(H)}}_{i,\bar{U}}}\prod_{\mathclap{j\in [t]\colon \{i,j\}\in \Ecal_{S_3^*}}} \min(x_ix_j,1)\prod_{j\in [t]\colon \{i,j\}\notin \Ecal_{S_3^*}}\min(1/(x_ix_j),1)\dd x_i .
	\end{align}
	Then~\eqref{eq:S3intind} can be bounded by
	\begin{align*}\label{eq:intSind}
		&\int_{1}^{\infty}\cdots \int_{1}^{\infty}\prod_{p \in [t]}x_p^{-\tau+\zeta_p}\ \prod_{\mathclap{u,v\in U\colon \{u,v\}\notin \Ecal_{S_3^*}}} \ \ \ \  \ \frac{1}{x_ux_v}
		\prod_{i=t+1}^k\tilde{h}(i,\boldsymbol{x})\dd x_t\cdots \dd x_1\\
		& =\int_{1}^{\infty}\cdots \int_{1}^{\infty}\prod_{p \in [t]}x_p^{-\tau+\zeta_p-(|U|-1-d^{\sss{(H)}}_{p,U})}\
		\prod_{i=t+1}^k\tilde{h}(i,\boldsymbol{x})\dd x_t\cdots \dd x_1\\
		& =\int_{1}^{\infty}\cdots \int_{1}^{\infty}\prod_{p \in [t]}x_p^{-\tau+\zeta_p-t+1+d^{\sss{(H)}}_{p,[t]}}\
		\prod_{i=t+1}^k\tilde{h}(i,\boldsymbol{x})\dd x_t\cdots \dd x_1.
	\end{align*}
	We can write $\tilde{h}(i,\boldsymbol{x})$ as
	\begin{align}
		\tilde{h}(i,\boldsymbol{x})& = \int_{0}^{1/x_t}x_i^{-\tau+\zeta_i+d^{\sss{(H)}}_{i,S_3^*}}\dd x_i\cdot \prod_{j=1}^t x_j^{\ind{\{i,j\}\in \mathcal{E}_{S_3^*}}}\nonumber\\
		& + \int_{1/x_t}^{1/x_{t-1}}x_i^{-\tau-1+\zeta_i+d^{\sss{(H)}}_{i,S_3^*}}\dd x_i\cdot  \prod_{j=1}^{t-1} x_j^{\ind{\{i,j\}\in \mathcal{E}_{S_3^*}} }\prod_{k=t}^t x_k^{-\ind{\{i,k\}\notin \mathcal{E}_{S_3^*}}}	+  \nonumber\\
		& + \int_{1/x_{t-1}}^{1/x_{t-2}}x_i^{-\tau+\zeta_i+d^{\sss{(H)}}_{i,S_3^*}-2}\dd x_i \prod_{j=1}^{t-2} x_j^{\ind{\{i,j\}\in \mathcal{E}_{S_3^*}}} \prod_{k=t-1}^t x_k^{-\ind{\{i,k\}\notin \mathcal{E}_{S_3^*}}}+\cdots \nonumber\\
		& + \int_{1/x_{1}}^1x_i^{-\tau+\zeta_i+d^{\sss{(H)}}_{i,S_3^*}-t}\dd x_i\cdot \prod_{k=1}^t x_k^{-\ind{\{i,k\}\notin \mathcal{E}_{S_3^*}}}.
	\end{align}
	By Lemma~\ref{lem:dmotif}, $\zeta_i+d^{\sss{(H)}}_{i,S_3^*}\geq 2$ for $i\in S_3^*$ so that the first integral is finite. 
	Computing these integrals yields
	\begin{align}
		h(i,\boldsymbol{x}) & = C_0\prod_{k=1}^t x_k^{-\ind{\{i,k\}\notin \mathcal{E}_{S_3^*}}}\nonumber\\
		& \quad  + C_1 x_{1}^{\tau-\zeta_i-d^{\sss{(H)}}_{i,S_3^*}+t-2}\prod_{j=1}^1 x_j^{\ind{\{i,j\}\in \mathcal{E}_{S_3^*}}}\prod_{k=2}^t x_k^{-\ind{\{i,k\}\notin \mathcal{E}_{S_3^*}}}\nonumber \\
		& \quad  + C_2 x_{2}^{\tau-\zeta_i-d^{\sss{(H)}}_{i,S_3^*}+t-3}\prod_{j=1}^2 x_j^{\ind{\{i,j\}\in \mathcal{E}_{S_3^*}}}\prod_{k=3}^t x_k^{-\ind{\{i,k\}\notin \mathcal{E}_{S_3^*}}}+\cdots  \nonumber\\
		& \quad +C_{t} x_{t}^{\tau-\zeta_i-d^{\sss{(H)}}_{i,S_3^*}-1}\prod_{j=1}^{t} x_j^{\ind{\{i,j\}\in \mathcal{E}_{S_3^*}}}\nonumber\\
		& {= :  C_0h_0(i,\boldsymbol{x}) +C_1 h_1(i,\boldsymbol{x}) +\dots+C_{t}h_{t}(i,\boldsymbol{x})},
	\end{align}	
	for some constants $C_0,\dots,C_{t}$. Assume that $i$ is connected to $l$ vertices in $U$, so that there are $l$ vertices in $\{1,2.\dots,t\}$ such that $\ind{\{i,t\}\in\mathcal{E}_{S_3^*}}=1$ and $t-l$ such that $\ind{\{i,t\}\notin\mathcal{E}_{S_3^*}}=1$. Then, 
	\begin{align*}
		\frac{h_{p+1}(i,\boldsymbol{x})}{h_{p}(i,\boldsymbol{x})}&=\frac{\prod_{j=1}^{p+1}x_j^{\ind{\{i,j\}\in \mathcal{E}_{S_3^*}}}x_{p+1}^{\tau-\zeta_i-d^{\sss{(H)}}_{i,S_3^*}+t-1-(p+1)}\prod_{k=p+2}^{t} x_k^{-\ind{\{i,k\}\notin \mathcal{E}_{S_3^*}}}}{\prod_{j=1}^{p}x_j^{\ind{\{i,j\}\in \mathcal{E}_{S_3^*}}}x_{p}^{\tau-\zeta_i-d^{\sss{(H)}}_{i,S_3^*}+t-1-p}\prod_{k=p+1}^{t} x_k^{-\ind{\{i,k\}\notin \mathcal{E}_{S_3^*}}}}\\
		&=\frac{x_{p+1}^{\ind{\{i,p+1\}\in \mathcal{E}_{S_3^*}}}x_{p+1}^{\tau-\zeta_i-d^{\sss{(H)}}_{i,S_3^*}+t-1-(p+1)}}{x_{p}^{\tau-\zeta_i-d^{\sss{(H)}}_{i,S_3^*}+t-1-p}x_{p+1}^{-\ind{\{i,p+1\}\notin \mathcal{E}_{S_3^*}}}}\\
		&=\Big(\frac{x_{p+1}}{x_{p}}\Big)^{\tau-\zeta_i-d^{\sss{(H)}}_{i,S_3^*}+t-p-1},
	\end{align*}
	which is larger than 1 for $p< \tau-\zeta_i-d^{\sss{(H)}}_{i,S_3^*}+t-2$ as $x_{p+1}>x_p$, and at most 1 for $p\geq\tau-\zeta_i-d^{\sss{(H)}}_{i,S_3^*}+t-2$. Thus, $p^*=p^*_i={\rm argmax}_{p} h_p(i,\boldsymbol{x})=\lfloor \tau-\zeta_i-d^{\sss{(H)}}_{i,S_3^*}+t-1\rfloor$.
	Therefore, there exists a $K>0$ such that
	\begin{align}
		\label{eq:hstar}
		h(i,\boldsymbol{x})\leq K h_{p^*_i}(i,\boldsymbol{x}).
	\end{align}
	
	For all $j\in U$, let 
	\begin{align}
		\label{eq:Q}
		Q_j^+ = \{i\in \bar{U}\colon \{i,j\}\in \Ecal_{S_3^*},{p_i^*}\geq j\}
	\end{align}
	denote the set of neighbors $i\in\bar{U}$ of $j\in U$ such that $x_j$ appears in $h_{p^*_i}(i,\boldsymbol{x})$ with exponent $+1$. (note that $i<j$ for all $i\in \bar{U}, j\in U$). Similarly, let
	\begin{align}
		\label{eq:Qmin}
		Q_j^- = \{i\in \bar{U}\colon \{i,j\}\notin \Ecal_{S_3^*},{p_i^*}< j\}
	\end{align}
	be the set of non-neighbors $i\in\bar{U}$ of $j$ such that $x_j$ appears in $h_{p^*_i}(i,\boldsymbol{x})$ with exponent $-1$. Furthermore, let ${W_j=\{i\in\bar{U}\colon {p^*_i}=j\}}$. Thus, the vertices in $W_j$ appear with exponent $\tau-\zeta_i-d^{\sss{(H)}}_{i,S_3^*}+t-1-j$ in $h_{p^*_i}(i,\boldsymbol{x})$. Then, by the definition of $\zeta_i$ in~\eqref{eq:zeta}
	\begin{align}
		\sum_{i\in W_j}\zeta_i+d^{\sss{(H)}}_{i,S_3^*}-t+j = 2E_{W_j}+E_{W_j,V_H\setminus W_j}+(j-t-|S_2^*|)|W_j|.
	\end{align}
	This yields
	\begin{equation}\label{eq:integrals}
		\begin{aligned}[b]
			&\int_{1}^{\infty}\cdots \int_{x_{t-1}}^{\infty}\prod_{j \in [t]}x_j^{-\tau+\zeta_j-t+1+d^{\sss{(H)}}_{j,[t]}} \prod_{i=t+1}^k\hat{h}(i,\boldsymbol{x})\dd x_t\cdots \dd x_1\\
			&\leq \tilde{K}\int_{1}^{\infty}\cdots \int_{x_{t-1}}^{\infty}\prod_{j \in [t]}x_j^{-\tau+\zeta_j-t+1+d^{\sss{(H)}}_{j,[t]}}\prod_{i=t+1}^k\hat{h}_{p^*_i}(i,\boldsymbol{x})\dd x_t\cdots \dd x_1\\
			& \leq \tilde{K}\int_{1}^{\infty} \int_{x_1}^{\infty}\cdots \int_{x_{t-1}}^{\infty}\prod_{j \in [t]}x_j^{-\tau+\zeta_j-t+1+d^{\sss{(H)}}_{j,[t]}+|Q_j^+|-|Q_j^-|+(\tau-1-j+t+|S_2^*|)\abs{W_j}-2E_{W_j}-E_{W_j,\hat{W}_j}}\dd x_{t}\cdots \dd x_1.
		\end{aligned}
	\end{equation}
	for some $\tilde{K}>0$, where $\hat{W}_j=V_H\setminus W_j$. 
	
	We will now use the uniqueness of the solution of the optimization problem in~\eqref{eq:maxeqalphaind} to prove that the integral over $x_t$ in~\eqref{eq:integrals} is finite. 
	First of all, 
	\begin{align}
		Q_t^+=\{i\in \bar{U}:\{i,t\}\in\mathcal{E}_{S_3^*},p_i^*=t\}=\{i\in W_t:\{i,t\}\in\mathcal{E}_{S_3^*}\},
	\end{align}
	so that $|Q_t^+|=d^{\sss{(H)}}_{t,W_t}$,
	whereas
	\begin{align}
		Q_t^-=\emptyset.
	\end{align}
	
	We will now prove that the exponent of $x_t$ in~\eqref{eq:integrals}
	$$-\tau+\zeta_t-t+1+d^{\sss{(H)}}_{t,[t]}+|Q_t^+|-|Q_t^-|+(\tau-1+|S_2^*|)\abs{W_t}-2E_{W_t}-E_{W_t,\hat{W}_t}<-1,$$
	or 
	\begin{equation}\label{eq:exponent}
		(\tau-1+|S_2^*|)\abs{W_t}+1-\tau-|S_2^*|-t-2E_{W_t}-E_{W_t,\hat{W}_t}+d^{\sss{(H)}}_{t}-d^{\sss{(H)}}_{t,\bar{U}\setminus W_t}<-1.
	\end{equation}
	Define $\hat{S}_2=\hat{S}_2^*\cup \{ t \}$, $\hat{S}_1=\hat{S}_1^*\cup {W}_t$ and $\hat{S}_3=S_3^*\setminus(W_t\cup \{t\})$. This gives 
	\begin{align}
		E_{\hat{S}_1}-E_{S_1^*}&=E_{W_t}+E_{W_t,S_1^*} \label{eq:comparesets1},\\
		E_{\hat{S}_1,\hat{S}_3}-E_{S_1^*,S_3^*}
		& =E_{W_t,S_3^*}-E_{W_t}-E_{W_t,S_1^*}-{|Q_t^+|}-d^{\sss{(H)}}_{t,S_1^*},\\
		E_{\hat{S}_2,\hat{S}_3}-E_{S_2^*,S_3^*}
		& =d^{\sss{(H)}}_{t,S_3^*}-d^{\sss{(H)}}_{t,S_2^*}-{|Q_t^+|}-E_{W_t,S_2^*},\\
		E_{\hat{S}_1,V_1}-E_{S_1^*,V_1}&=E_{W_t,V_1},\\
		E_{\hat{S}_2,V_1}-E_{S_2^*,V_1}&=d^{\sss{(H)}}_{t,V_1},\label{eq:comparesets4}	\\
		E_{\hat{S}_2}-E_{S_2^*}&=d^{\sss{(H)}}_{t,S_2^*} 
	\end{align}
	Because~\eqref{eq:maxeqsub} is \emph{uniquely} optimized by $S_1^*$, $S_2^*$ and $S_3^*$, we obtain
	\begin{align}
		&	\abs{S_1^*}+\frac{2-\tau-k+|S_1^*|+k_1}{\tau-1}\abs{S_2^*}-\frac{2E_{S_1^*}-2E_{S_2^*}+E_{S_1^*,S_3^*}-E_{S_2^*,S_3^*}+E_{S_1^*,V_1}-E_{S_2^*,V_1}}{\tau-1}\nonumber\\
		&	> \abs{\hat{S}_1}+\frac{2-\tau-k+|\hat{S}_1|+k_1}{\tau-1}\abs{\hat{S}_2}-\frac{2E_{\hat{S}_1}-2E_{\hat{S}_2}+E_{\hat{S}_1,\hat{S}_3}-E_{\hat{S}_2,\hat{S}_3}+E_{\hat{S}_1,V_1}-E_{\hat{S}_2,V_1}}{\tau-1}
	\end{align}
	Plugging in~\eqref{eq:comparesets1}--\eqref{eq:comparesets4} and using that $k-k_1-|S_1^*|=|S_2^*|+|S_3^*|$ yields
	\begin{align}
		&|W_t|+\frac{2-\tau}{\tau-1}+\frac{|S_2^*||W_t|-|S_2^*|-|S_3^*|+|W_t|}{\tau-1}\nonumber\\
		& -\frac{E_{W_t}+E_{W_t,S_1^*}+E_{W_t,V_1}+E_{W_t,S_2^*}+E_{W_t,S_3^*}-d^{\sss{(H)}}_{t,S_2^*}-d^{\sss{(H)}}_{t,S_1^*}-d^{\sss{(H)}}_{t,S_3^*}-d^{\sss{(H)}}_{t,V_1}}{\tau-1}< 0.
	\end{align}
	Multiplying by $\tau-1$ then gives
	\begin{align}
		&(\tau-1)|W_t|+1-\tau+|S_2^*||W_t|-|S_2^*|-|S_3^*|+|W_t|-{2E_{W_t}-E_{W_t,\hat{W}_t}+d^{\sss{(H)}}_{t}}<-1.
	\end{align}
	Using that $|W_t|\leq |S_3^*|-t$ then yields
	\begin{align}
		&(\tau-1)|W_t|+1-\tau+|S_2^*||W_t|-|S_2^*|-t-{2E_{W_t}-E_{W_t,\hat{W}_t}+d^{\sss{(H)}}_{t}}<-1,
	\end{align}
	so that~\eqref{eq:exponent} also holds.
	
	Thus, the integral in~\eqref{eq:integrals} over $x_t$ results in a power of $x_{t-1}$. 
	We can then use a similar technique to show that the power of $x_{t-1}$ is also smaller than one, and iterate to finally show that the integral in~\eqref{eq:integrals} is finite, so that~\eqref{eq:S3intind} is also finite.
\end{proof}

\begin{proof}[Proof of Lemma~\ref{lem:S1S2intind}]
	This lemma can be proven along similar lines of Lemma~\ref{lem:S3intind}. In particular, it follows the same strategy and computations as in~\cite[Lemma 7.3]{hofstad2017d}, where the factors $x_ix_i/(x_ix_j+1)$ and $L_n/(x_ix_j+1)$ are bounded by terms of $\min(x_ix_j,1)$ and $\min(1/(x_ix_j),1)$ as in the proof of Lemma~\ref{lem:S3intind}. 
\end{proof}

\end{document}